\documentclass[reqno,centertags,12pt]{amsart}
\usepackage{amsmath,amsthm,amscd}\usepackage{amssymb,verbatim,amssymb}
\usepackage{amsfonts,amscd, graphicx, dsfont}
\usepackage{lmodern}
\usepackage{hyperref}
\hypersetup{
    colorlinks=true,
    linkcolor=blue,
    filecolor=blue,
    urlcolor=blue,
    pdftitle={Sharelatex Example},
    citecolor = blue,
}
\usepackage{remreset}
\makeatletter\@removefromreset{footnote}{chapter}\makeatother 
\usepackage[english]{babel}
\usepackage{mathrsfs}
\usepackage[top=20mm, bottom=20mm, left=24mm, right=24mm]{geometry}
\numberwithin{equation}{section}

\newtheorem{theorem}{Theorem}

\newtheorem{proposition}{Proposition}

\newtheorem{lemma}{Lemma}

\newtheorem{remark}{Remark}

\newtheorem{definition}{Definition}

\usepackage{enumerate}
\def\i1{\mathds{1}}
\def\h1{\hspace{0.2cm}}
\def\v1{\vskip0.2cm}

\def\R{\mathbb{R}}

\begin{document}
\title{Wiener integrals with respect to the two-parameter tempered Hermite random fields}
\author{Atef Lechiheb}
\address{University of Sousse - Higher Institute of Transport and Logistics of Sousse\\
\indent Lab: Analyse, probabilités et fractals LR18ES17 - Faculty of Science of Monastir}
\email{atef.lechiheb@gmail.com}
\dedicatory{}
\date{Mars 2022}

\begin{abstract}
The two-parameter tempered Hermite field modifies the power law kernel in the moving average representation of the Hermite field by adding an exponential tempering. This paper develops the basic theory of two-parameter tempered Hermite field, including moving average, sample path properties, spectral representations and the theory of Wiener stochastic integration with respect to the two-parameter tempered Hermite field of order one.
\end{abstract}
\maketitle

\section{Introduction}
Let $\big(W(x,y),\,\,x,y\in\R\big)$ be a two-parameter Brownian field (see Definition \ref{BF} below). The two-parameter Hermite random fields or Hermite sheets of order $k\geq 1$ are stochastic processes defined as multiple Wiener-It\^{o} integrals of order $k$ with respect to $W$.
\begin{align}\label{H}
 Z^{k,H_1,H_2}(t,s):=& \int\limits_{(\mathbb{R}^2)^k}^\prime dW(x_1,y_1)\ldots dW(x_k,y_k) \\
\nonumber&\quad \int\limits_0^t\int\limits_0^s \prod_{j=1}^k (u-x_j)_+^{-(\frac{1}{2}+\frac{1-H_1}{k})} (v-y_j)_+^{-(\frac{1}{2}+\frac{1-H_2}{k})} \,du\,dv, \
\end{align}
where $s,\,t\in\mathbb{R}^2$ and $H=(H_1,H_2)\in\big(\frac{1}{2},1\big)^2$ (the prime $^\prime$ on the integral indicates that one does not integrate on the hyperplanes $(x_i,y_i)=(x_j,y_j)$, $i\neq j$). Hermite fields are self-similar processes with stationary increments (see \cite{HS3} for more details).

From expression \eqref{H}, it is possible to note that for one parameter case, we recover the Hermite process which represents a family that has been studied by several authors see, e.g., \cite{taqq}, \cite{Taqqu79} and \cite{Taqqu17}.

F. \emph{Sabzikar} \cite{Sab} has introduced a new class of one parameter stochastic processes, called tempered Hermite process. He has modified the kernel of the one-parameter Hermite process $Z^{k,\,H}$ multiplying it by an exponential tempering factor $\lambda>0$. This process has the following time domain representation
\begin{equation}\label{thp}
  Z^{k,\,H}_\lambda(t):= \int\limits_{\mathbb{R}^k}^{\prime} \int\limits_0^t \prod_{j=1}^k\Big((s-y_j)_+^{-(\frac{1}{2}+\frac{1-H}{k})} \,e^{-\lambda(s-y_j)_+}\Big)\,ds\,B(dy_1)\ldots B(dy_k),
\end{equation}
where $B=\{B(t),\, \, t\in\mathbb{R}\}$ is a real-valued standard Brownian motion, $H>\frac{1}{2}$ and $\lambda>0$. It has been shown that this process has stationary increments but it is not self-similar.

The natural question in the present work is how to extend this class of processes to the two-parameter case and how to build the process that will be called the ``two-parameter tempered Hermite field" defined either as a natural extension of the tempered Hermite process \eqref{thp} to two dimensional random fields or as a modification of the kernel of the Hermite field \eqref{H} multiplying it by an exponential factor $\lambda=(\lambda_1,\lambda_2)\in (0,\infty)^2$ such that this random field is well defined for Hurst parameter $H=(H_1,H_2)\in (\frac{1}{2},\infty)^2$.

The remainder of the paper is organized as follows. In Section \ref{sect2} we recall the concept of multiple Wiener-It\^{o} integrals with respect to Brownian field and we present its properties. In Section \ref{sect3}, we introduce the main objective of this study which is the so-called tempered Hermite field, and derive some of its basic properties. In section \ref{sect4}, we study the Hermitian random measures on $(\mathbb{R}^2,\mathcal{B}(\mathbb{R}^2))$ and we give the spectral representation of the two-parameter tempered Hermite field. In Section \ref{sect5} we prove some basic results on the two-parameter tempered
fractional calculus, which will be needed in the sequel. Finally, in Section \ref{sect6} we apply the results of Section \ref{sect5} to construct a suitable theory of stochastic integration for two-parameter tempered Hermite field of order one.

\section{Multiple Wiener-It\^{o} integrals with respect to the Brownian field}
\label{sect2}
In this section, we briefly review the theory of multiple Wiener-It\^{o} integrals with respect to the Brownian field. For more details, we refer the reader to \cite{Nua} and \cite{Pak}. Let us first recall the definition of the standard Brownian field.
\begin{definition}\label{BF}
The two-parameter standard Brownian field is the centered Gaussian process $\big\{\,W(x,y):\,\, x,\,y\in\R\, \big\}$ such that $W(0,0)=0$ and its covariance function is given by
\begin{equation*}
\mathbb{E}\big[W(s,t)W(u,v)\big] \,=\, (s\wedge u)(t\wedge v).
\end{equation*}
\end{definition}
We can now introduce the multiple Wiener-It\^{o} integrals with respect to the Brownian field. Let $f:(\mathbb{R}^2)^k\to\mathbb{R}$ be a deterministic function and let us denote by $I_k^{W}(f)$ the $k$-fold multiple Wiener-It\^{o} integrals of $f$ with respect to the standard two-sided Brownian field $(W(x,y))_{x,y\in\mathbb{R}}$. This integral has the following form:
\begin{align}\label{af}
 I_k^{W}(f) &= \int_{(\mathbb{R}^2)^k}^\prime f\left((x_1,y_1),(x_2,y_2),\ldots,(x_k,y_k)\right)\\
 &\quad\times dW(x_1,y_1)dW(x_2,y_2)\ldots dW(x_k,y_k).\nonumber
\end{align}
The actual definition first defines $I_k^{W}(f)$ for elementary functions in a natural way, and then extends $I_k^{W}(f)$ to $f\in L^2\big((\mathbb{R}^2)^k\big)$ so that the following properties hold:
\begin{itemize}
  \item $I_k^{W}$ is linear,
  \item $I_k^{W}(f)\,=\, I_k^{W}(\tilde{f})$, where $\tilde{f}$ is the symmetrization of $f$ defined by
\begin{equation*}
\tilde{f}((x_1,y_1),\ldots,(x_k,y_k))=\frac{1}{k!}\sum_{\sigma}f((x_{\sigma(1)},y_{\sigma(1)}),\dots(x_{\sigma(k)},y_{\sigma(k)})),
\end{equation*}
$\sigma$ running over all permutations of $\{1,\ldots,k\}$,
  \item multiple Wiener integrals satisfy the following isometry and orthogonality properties
$$\mathbb{E}\big[I_k^{W}(f)I_{k'}^{W}(g)\big]=\begin{cases}k!\langle\widetilde{f},\widetilde{g}\rangle_{L^2\big((\mathbb{R}^2)^k\big)}& \text{if $k=k'$}\\0& \text{if $k\neq k'$},\end{cases}$$
where $\langle\widetilde{f},\widetilde{g}\rangle_{L^2\big((\mathbb{R}^2)^k\big)}$ indicates the standard inner product in $L^2\big((\mathbb{R}^2)^k\big)$.
\end{itemize}
The prime $^\prime$ on the integral \eqref{af} indicates that one does not integrate on the hyperplanes $(x_i,y_i)=(x_j,y_j)$, $i\neq j$. This ensures that $\mathbb{E}\big[I_k^{W}(f)\big]=0$.

Next, we will present the generalized stochastic Fubini theorem with respect to the two-parameter standard Brownian field. This theorem states that one can interchange Lebesgue integrals and multiple Wiener-It\^{o} stochastic integrals with respect to the Brownian field.\\
Let $k\in\mathbb{N}$, the mixed Lebesgue space and its norm of a function $f:\R^2\times (\R^2)^k\to\R$ are:
\begin{equation*}
  \|f\|_{p_1,p_2}\,=\, \Big(\int_{\R^2}\left(\int_{\R^{2k}} \vert f(a,b, \mathbf{u}_1,\ldots \mathbf{u}_k)\vert ^{p_1}\,d\mathbf{u}_1\ldots d\mathbf{u}_k\right)^{\frac{p_2}{p_1}}\,da\,db\Big)^{\frac{1}{p_2}},
\end{equation*}
\begin{equation*}
  \mathcal{L}_{p_1,p_2}(\R^2\times \R^{2k})=\Big\{f:\R^2\times (\R^2)^k\to\R,\,\text{Borelian},\, \|f\|_{p_1,p_2}<\infty\Big\}.
\end{equation*}
Let us remark that if $f \in\mathcal{L}_{1,2}(\R^{2k}\times \R^2)$ using the Cauchy-Schwartz' inequality:
\begin{eqnarray*}
  \|f\|_{1,2}^2 &=& \Big(\int _{\R^{2k}}\Big(\int_{\R^2} \vert f(a,b, \mathbf{u}_1,\ldots \mathbf{u}_k)\vert\,dadb\Big)^2 \,d\mathbf{u}_1\ldots d\mathbf{u}_k\Big) \\
   &=& \int_{\R^2}\int_{\R^2}\int_{\R^{2k}}\vert f(a_1,b_1, \mathbf{u}_1,\ldots \mathbf{u}_k)\vert f(a_2,b_2, \mathbf{u}_1,\ldots \mathbf{u}_k)\vert \\
   &&  \qquad \qquad \,d\mathbf{u}_1\ldots d\mathbf{u}_k\,da_1db_1da_2db_2
   \\
   &\leq& \|f\|^2_{2,1},
\end{eqnarray*}
this yields the inclusion $\mathcal{L}_{2,1}(\R^2\times \R^{2k})\subset\mathcal{L}_{1,2}(\R^{2k}\times \R^2)$.
\begin{theorem}
Let $f\in\mathcal{L}_{2,1}(\R^2\times \R^{2k})$ and $(W(x,y),\,x,\,y\in\R)$ be a two-parameter standard Brownian field. Then almost surely:
\begin{align}
&\int_{\R^2}\Big(\int_{\R^{2k}}  f\big(a,b, (x_1,y_1),\ldots, (x_k,y_k)\big)\,dW(x_1,y_1)\ldots dW(x_k,y_k)\Big)\,da\,db\nonumber\\
&\quad =\, \int_{\R^{2k}}\Big(\int_{\R^2} f\big(a,b, (x_1,y_1),\ldots, (x_k,y_k)\big)\,da\,db\Big)dW(x_1,y_1)\ldots dW(x_k,y_k).
\end{align}
\end{theorem}
\begin{proof}
The proof of this theorem is similar to that of \cite[Theorem 2.1]{taqq} where the function is defined on $\R^2\times(\mathbb{R}^2)^k$.\\
The map $$Y_1:f\mapsto \int_{\R^2}\Big(\int_{\R^{2k}}  f\big(a,b, (x_1,y_1),\ldots, (x_k,y_k)\big)\,dW(x_1,y_1)\ldots dW(x_k,y_k)\Big)\,da\,db$$ is a continuous linear map on the step functions in $\mathcal{L}_{2,1}(\R^2\times \R^{2k})$ taking its values in $L^2(\Omega)$. The set of these step functions is dense in $\mathcal{L}_{2,1}(\R^2\times \R^{2k})$  so this map admits a unique continuous linear extension on $\mathcal{L}_{2,1}(\R^2\times \R^{2k})$. \\
Let the map $$Y_2: f\mapsto \int_{\R^{2k}}\Big(\int_{\R^2} f\big(a,b, (x_1,y_1),\ldots, (x_k,y_k)\big)\,da\,db\Big)dW(x_1,y_1)\ldots dW(x_k,y_k).$$
It is a linear continuous map on $\mathcal{L}_{2,1}(\R^2\times \R^{2k})\subset\mathcal{L}_{1,2}(\R^{2k}\times \R^2)$ with a norm 1 from $\mathcal{L}_{1,2}(\R^{2k}\times \R^2)$ to $L^2(\Omega)$ so `a fortiori' on $\mathcal{L}_{2,1}(\R^2\times \R^{2k})$:
{\small\begin{eqnarray*}
\displaystyle\|Y_2\|_2^2&=& \int_{\R^{2k}}\Big(\int_{\R^2} f\big(a,b, (x_1,y_1),\ldots, (x_k,y_k)\big)\,da\,db\Big)^2dW(x_1,y_1)\ldots dW(x_k,y_k)\\
&=&\|f\|^2_{1,2}\\
&\leq& \|f\|^2_{2,1}.
\end{eqnarray*}}
Finally, the maps $Y_i$, $i = 1,\, 2$, are well defined and coincide on the step functions.
\end{proof}
\section{Two-parameter tempered Hermite field}
\label{sect3}
Now, we are going to introduce the main object of this paper: the two-parameter tempered Hermite random field or tempered Hermite sheet. We give its definition and derive its basic properties. We give by the following lemma which states that our process in Definition \ref{defni} below is well defined.
\begin{lemma}\label{lemmma}Let $k\in\mathbb{N}^\ast$, $H_1,H_2>1/2$  and $\lambda_1,\lambda_2>0$. The function
\begin{align}
&\qquad h_{s,t}^{H_1,H_2,\lambda_1,\lambda_2}((x_1,y_1),\dots,(x_k,y_k))\nonumber\\
&=\, \int_0^t\int_0^s\prod_{j=1}^k (a-x_j)_+^{-(\frac{1}{2}+\frac{1-H_1}{k})}e^{-\lambda_1(a-x_j)_+}(b-y_j)_+^{-(\frac{1}{2}+\frac{1-H_2}{k})}e^{-\lambda_2(b-y_j)_+} \,da\,db
\end{align}
is well defined in $L^2\big((\mathbb{R}^2)^ k\big)$.
\end{lemma}
\begin{proof}
The proof is similar to that of \cite[Theorem 3.5]{Bai-Taqqu} and \cite[Lemma 1]{Sab}. \\
To show that $h_{s,t}^{H_1,H_2,\lambda_1,\lambda_2}((x_1,y_1),\dots,(x_k,y_k))$ is square integrable over $(\R^2)^k$, we write
{\small\begin{align*}
&\qquad \int_{(\R^2)^k}h_{s,t}^{H_1,H_2,\lambda_1,\lambda_2}((x_1,y_1),\dots,(x_k,y_k))^2\,dx_1dy_1\ldots dx_kdy_k\\
&=\int_{(\R^2)^k}\Bigg[ \int_0^t\int_0^s\int_0^t\int_0^s\prod_{j=1}^k (a_1-x_j)_+^{-(\frac{1}{2}+\frac{1-H_1}{k})}e^{-\lambda_1(a_1-x_j)_+}(b_1-y_j)_+^{-(\frac{1}{2}+\frac{1-H_2}{k})}\\
&\quad \times e^{-\lambda_2(b_1-y_j)_+} (a_2-x_j)_+^{-(\frac{1}{2}+\frac{1-H_1}{k})}e^{-\lambda_1(a_2-x_j)_+}(b_2-y_j)_+^{-(\frac{1}{2}+\frac{1-H_2}{k})}e^{-\lambda_2(b_2-y_j)_+} \\
&\qquad \qquad \qquad da_1db_1da_2db_2\Bigg] dx_1dy_1\ldots dx_kdy_k\\
&= 2^2 \int_0^t\,da_1\int_{a_1}^tda_2\int_0^s\,db_1\int_{b_1}^sdb_2\Bigg[\int_{(\R^2)^k} \prod_{j=1}^k (a_1-x_j)_+^{-(\frac{1}{2}+\frac{1-H_1}{k})}e^{-\lambda_1(a_1-x_j)_+} \\
&\quad\times (b_1-y_j)_+^{-(\frac{1}{2}+\frac{1-H_2}{k})}e^{-\lambda_2(b_1-y_j)_+} (a_2-x_j)_+^{-(\frac{1}{2}+\frac{1-H_1}{k})}e^{-\lambda_1(a_2-x_j)_+}\\
&\quad \times(b_2-y_j)_+^{-(\frac{1}{2}+\frac{1-H_2}{k})}e^{-\lambda_2(b_2-y_j)_+} dx_1dy_1\ldots dx_kdy_k\Bigg]\\
&= 2^2 \int_0^t\,du_1\int_0^{t-u_1}du_2\int_0^s\,dv_1\int_0^{s-v_1}dv_2\Bigg[\int_{(\R^2)^k}\prod_{j=1}^k (\xi_j)_+^{-(\frac{1}{2}+\frac{1-H_1}{k})}\\
&\quad \times e^{-\lambda_1(\xi_j)_+} (\omega_j)_+^{-(\frac{1}{2}+\frac{1-H_2}{k})}e^{-\lambda_2(\omega_j)_+} (\xi_j+u_2)_+^{-(\frac{1}{2}+\frac{1-H_1}{k})}\\
&\quad\times e^{-\lambda_1(\xi_j+u_2)_+}(\omega_j+v_2)_+^{-(\frac{1}{2}+\frac{1-H_2}{k})}e^{-\lambda_2(\omega_j+v_2)_+} d\xi_1d\omega_1\ldots d\xi_kd\omega_k\Bigg]\\
&\qquad(u_1=a_1,\,\,u_2=a_2-a_1,\,\, v_1=b_1,\,\, v_2=b_2-b_1,\,\,\xi_j= a_1-x_j,\,\,\omega_j=b_1-y_j).
\end{align*}}

Then, {\small
\begin{align*}
&\qquad \int_{(\R^2)^k}h_{s,t}^{H_1,H_2,\lambda_1,\lambda_2}((x_1,y_1),\dots,(x_k,y_k))^2\,dx_1dy_1\ldots dx_kdy_k\\
&=2^2 \int_0^t\,du_1\int_0^{t-u_1}e^{-k\lambda_1u_2}du_2 \Big[\int_{\R^+} \xi^{-(\frac{1}{2}+\frac{1-H_1}{k})} (\xi+u_2)_+^{-(\frac{1}{2}+\frac{1-H_1}{k})} e^{-2\lambda_1 \xi}\,d\xi\Big]^k\\
&\quad \times \int_0^s\,dv_1\int_0^{s-v_1}e^{-k\lambda_2 v_2}dv_2 \Big[\int_{\R^+} \omega^{-(\frac{1}{2}+\frac{1-H_2}{k})} (\omega+v_2)_+^{-(\frac{1}{2}+\frac{1-H_2}{k})} e^{-2\lambda_2 \omega}\,d\omega\Big]^k\\
&=2^2 \int_0^t\,du_1\int_0^{t-u_1}e^{-k\lambda_1u_2}u_2^{2H_1-2}du_2 \Big[\int_{\R^+} x^{-(\frac{1}{2}+\frac{1-H_1}{k})} (x+u_2)_+^{-(\frac{1}{2}+\frac{1-H_1}{k})} e^{-2\lambda_1 x u_2}\,dx\Big]^k\\
&\quad \times \int_0^s\,dv_1\int_0^{s-v_1}e^{-k\lambda_2 v_2}v_2^{2H_2-2}dv_2 \Big[\int_{\R^+} y^{-(\frac{1}{2}+\frac{1-H_2}{k})} (y+v_2)_+^{-(\frac{1}{2}+\frac{1-H_2}{k})} e^{-2\lambda_2 y v_2}\,dy\Big]^k\\
&=2^2 \int_0^t\,du_1\int_0^{t-u_1}e^{-k\lambda_1u_2}u_2^{2H_1-2}du_2 \Big[\frac{\Gamma(\frac{1}{2}-\frac{1-H_1}{k})}{\sqrt{\pi}} \Big(\frac{1}{2\lambda_1u_2}\Big)^{\frac{H_1-1}{k}} e^{\lambda_1u_2}K_{\frac{1-H}{k}}(\lambda_1u_2)\Big]^k\\
&\quad \times \int_0^s\,dv_1\int_0^{s-v_1}e^{-k\lambda_2 v_2}v_2^{2H_2-2}dv_2 \Big[\frac{\Gamma(\frac{1}{2}-\frac{1-H_2}{k})}{\sqrt{\pi}} \Big(\frac{1}{2\lambda_1v_2}\Big)^{\frac{H_2-1}{k}} e^{\lambda_2v_2}K_{\frac{1-H}{k}}(\lambda_2v_2)\Big]^k\\
&= 2^2 \Big[\frac{\Gamma(\frac{1}{2}-\frac{1-H_1}{k})}{\sqrt{\pi} (2\lambda_1)^{\frac{H_1-1}{k}}}\Big]^k  \Big[\frac{\Gamma(\frac{1}{2}-\frac{1-H_2}{k})}{\sqrt{\pi} (2\lambda_2)^{\frac{H_2-1}{k}}}\Big]^k \int_0^t\,du_1\int_0^{t-u_1} \Big[u_2^{\frac{H_1-1}{k}}K_{\frac{1-H_1}{k}}(\lambda_1 u_2)\Big]^k\,du_2\\
&\quad \times \int_0^s\,dv_1\int_0^{s-v_1} \Big[v_2^{\frac{H_2-1}{k}}K_{\frac{1-H_2}{k}}(\lambda_2 v_2)\Big]^k\,dv_2\\
&= 2^2 \Big[\frac{\Gamma(\frac{1}{2}-\frac{1-H_1}{k})}{\sqrt{\pi} 2^{\frac{H_1-1}{k}}(\lambda_1)^{2\frac{H_1-1}{k}}}\Big]^k  \Big[\frac{\Gamma(\frac{1}{2}-\frac{1-H_2}{k})}{\sqrt{\pi} 2^{\frac{H_2-1}{k}}(\lambda_2)^{2\frac{H_2-1}{k}}}\Big]^k\\
&\quad \times \int_0^t\,du_1\int_0^{\lambda_1(t-u_1)} \Big[z_1^{\frac{H_1-1}{k}}K_{\frac{1-H_1}{k}}(z_1)\Big]^k\,dz_1\\ &\quad \times \int_0^s\,dv_1\int_0^{\lambda_2(s-v_1)} \Big[z_2^{\frac{H_2-1}{k}}K_{\frac{1-H_2}{k}}(z_2)\Big]^k\,dz_2
\end{align*}}
where we have applied the following integral formula
\begin{equation*}
\int_0^\infty x^{\nu-1}(x+\beta)^{\nu-1}e^{-\mu x}\,dx\,=\,\frac{1}{\sqrt{\pi}}\Bigg(\frac{\beta}{\mu}\Bigg)^{\nu-\frac{1}{2}}e^{\frac{\beta\mu}{2}}\Gamma(\nu)K_{\frac{1}{2}-\nu}\Bigg(\frac{\beta\mu}{2}\Bigg)
\end{equation*}
for $\vert \arg \beta\vert<\pi$, $\text{Re}\,\mu>0$, $\text{Re}\,\nu>0$. Here $K_\nu(x)$ is the modified Bessel function of the second kind (see, e.g., \cite[Section 9.6]{kind} or \cite[Section 11.5]{kind1}).\\
To finish the proof of our lemma, it suffices to show that
\begin{equation*}
\int_0^t\,du_1\int_0^{\lambda_1(t-u_1)} \Big[z_1^{\frac{H_1-1}{k}}K_{\frac{1-H_1}{k}}(z_1)\Big]^k\,dz_1
\end{equation*}
and
\begin{equation*}
\int_0^s\,dv_1\int_0^{s-v_1} \Big[v_2^{\frac{H_2-1}{k}}K_{\frac{1-H_2}{k}}(\lambda_2 v_2)\Big]^k\,dv_2
\end{equation*}
are finite for every $\lambda_1,\,\lambda_2>0$ and $H_1,\,H_2>\frac{1}{2}$.\\
First, assume $\frac{1}{2} < H_1,\,H_2 < 1$. In that case, $K_{\frac{1-H_1}{k}}(z_1)\sim z_1^{\frac{H_1-1}{k}}$ as $z_1\to0$ and $K_{\frac{1-H_2}{k}}(z_2)\sim z_2^{\frac{H_2-1}{k}}$ as $z_2\to0$ (see \cite[Chapter 9]{kind}), and hence the integrands $\Big[z_1^{\frac{H_1-1}{k}}K_{\frac{1-H_1}{k}}(z_1)\Big]^k \sim z_1^{2H_1-2}$ as $z_1\to0$ and $\Big[z_2^{\frac{H_2-1}{k}}K_{\frac{1-H_2}{k}}(z_2)\Big]^k \sim z_2^{2H_2-2}$ as $z_2\to0$, which are integrable provided that $H_1,\,H_2> \frac{1}{2}$ .\\
Now, let $H_1,H_2 > 1$. In the latter case, $K_{\frac{1-H_1}{k}}(z_1)\sim z_1^{\frac{1-H_1}{k}}$ as $z_1\to0$ and $K_{\frac{1-H_2}{k}}(z_2)\sim z_2^{\frac{1-H_2}{k}}$ as $z_2\to0$ and therefore the integrands $\Big[z_1^{\frac{H_1-1}{k}}K_{\frac{1-H_1}{k}}(z_1)\Big]^k \sim C_1$ as $z_1\to0$ and $\Big[z_2^{\frac{H_2-1}{k}}K_{\frac{1-H_2}{k}}(z_2)\Big]^k \sim C_2$ as $z_2\to0$, $C_1$ and $C_2$ are constants, which are integrable and this completes the proof.
\end{proof}
 Based on Lemma \ref{lemmma} and the expression \eqref{H} which describes the two-parameter Hermite field and the expression \eqref{thp} of the tempered Hermite process, we can introduce the following definition:
\begin{definition}\label{defni}
Let $k\in\mathbb{N}^\ast$, $H=(H_1,H_2)\in(1/2,\infty)^2$  and $\lambda=(\lambda_1,\lambda_2)\in(0,\infty)^2$. The random field
\begin{align}
\label{tHs}& Z^{k,\,H_1,H_2}_{\lambda_1,\lambda_2}(s,t)=\int_{(\mathbb{R}^2)^ k}^\prime\, dW(x_1, y_1)\ldots dW(x_k, y_k)\\
&\hspace{-.2em}\times \Bigg(\int_0^t \,da \int_0^s\,db \prod_{j=1}^k (a-x_j)_+^{-(\frac{1}{2}+\frac{1-H_1}{k})}e^{-\lambda_1(a-x_j)_+}(b-y_j)_+^{-(\frac{1}{2}+\frac{1-H_2}{k})}e^{-\lambda_2(b-y_j)_+}\Bigg),\nonumber
\end{align}
where $x_+=xI(x>0)$ and $W$ is a standard two-sided two-parameter Brownian field, is called a two-parameter tempered
Hermite field of order $k$. The prime $^\prime$ on the integral  indicates that one does not integrate on the hyperplanes $(x_i,y_i)=(x_j,y_j)$, $i \neq j$.
\end{definition}
The above integral \eqref{tHs} represents a multiple Wiener-It\^{o} integrals of order $k$ with respect to the standard two-sided two-parameter Brownian sheet $W$. For $k=1$, we call \eqref{tHs} a two-parameter tempered fractional Brownian sheet with Hurst multi-index $H=(H_1,H_2)$, for $k\geq 2$ the random field $Z^{k,\,H_1,H_2}_{\lambda_1,\lambda_2}(s,t)$ is not Gaussian and for $k=2$ we denominate it the two-parameter tempered Rosenblatt field.
Note that, when $\lambda_1=\lambda_2=0$  and the Hurst index $H$ satisfies $\frac{1}{2}<H_1,H_2<1$, then the integral \eqref{tHs} is simply a two-parameter Hermite field of order $k$, given in \eqref{H}, which is first introduced as a limit of some weighted Hermite variations of the fractional Brownian field (see \cite{HS1,HS2}) and then in \cite{HS3} this process has been defined as a multiple integral with respect to the standard Brownian field.


Next, we will prove the basic properties of the two-parameter tempered Hermite field: self-similarity, stationarity
of the increments, H\"{o}lder continuity and then compute the covariance of this processes.

Let us first recall the concepts of self-similarity and stationarity of increments for two-parameter stochastic process.
\begin{definition}\cite[Appendix A.2]{book-Tudor}
A two-parameter stochastic process {\small$(X(s,t))_{s,t \in T}$}, $T\subset\mathbb{R}^2$,
\begin{enumerate}
\item  is called self-similar with the self-similarity order $(\alpha,\beta)$ if for any $h,\,k>0$ the process
$$\widehat{X}(s,t):= h^\alpha k^\beta X\big(\frac{s}{h},\frac{t}{k}\big), \quad (s,t)\in T$$
has the same finite-dimensional distributions as the process $X$.
\item is said to be stationary if for
every integer $n \geq1$ and $(s_i, t_j )\in T$, $i,j=1,\ldots,n$, the distribution of the random
vector
$$\Big(X(s+s_1,t+t_1),X(s+s_2,t+t_2),\ldots ,X(s+s_n,t+t_n)\Big)$$
does not depend on $(s,t)$, where $s,t\geq0$, $(s +s_i, t + t_i )\in T$, $i=1,\ldots,n$.
\item has stationary increments if for every $h, k > 0$ the process
$$\Big(X(t+h,s+k)-X(t,s+k)-X(t+h,s) + X(t,s)\Big)_{(s,t)\in\R^2}$$
is stationary.
\end{enumerate}
\end{definition}
The following results show that the two-parameter tempered Hermite field has stationary increments but is not a self-similar process.
\begin{proposition}\label{sta}Le $k\in\mathbb{N}^\ast,\,H_1,\, H_2>\frac{1}{2}$ and $\lambda_1,\,\lambda_2>0$. The process $Z^{k,\,H_1,H_2}_{\lambda_1,\lambda_2}$ given by \eqref{tHs} has stationary increments such that
\begin{equation*}
\Big\{Z^{k,\,H_1,H_2}_{\lambda_1,\lambda_2}(h_1t,h_2s)\Big\}_{s,t\in\R} \,\overset{(d)}{=}\, \Big\{h_1^{H_1}h_2^{H_2}Z^{k,\,H_1,H_2}_{h_1\lambda_1,h_2\lambda_2}(s,t)\Big\}_{s,t\in\R}
\end{equation*}
for any scales factor $h_1,\,h_2 > 0$. Thus, the two-parameter tempered Hermite field is not self-similar. Here, the symbol $\overset{(d)}{=}$ indicates the equivalence of finite-dimensional distributions.
\end{proposition}
\begin{proof}
For every $h_1,\,h_2 > 0$, we have
\begin{align}
&\quad Z^{k,\,H_1,H_2}_{\lambda_1,\lambda_2}(h_1t,h_2s)=\int_{(\mathbb{R}^2)^ k}^\prime\, dW(x_1, y_1)\ldots dW(x_k, y_k) \Bigg(\int_0^{h_1t} \,da \int_0^{h_2s}\,db\nonumber\\
&\quad \times \prod_{j=1}^k (a-x_j)_+^{-(\frac{1}{2}+\frac{1-H_1}{k})}e^{-\lambda_1(a-x_j)_+}(b-y_j)_+^{-(\frac{1}{2}+\frac{1-H_2}{k})}e^{-\lambda_2(b-y_j)_+}\Bigg)\nonumber\\
& =h_1h_2\int_{(\mathbb{R}^2)^ k}^\prime\, dW(h_1x_1, h_2y_1)\ldots dW(h_1x_k, h_2y_k) \Bigg(\int_0^{t} \,da \int_0^{s}\,db \nonumber\\
&\quad\times \prod_{j=1}^k \big(ah_1-h_1x_j\big)_+^{-(\frac{1}{2}+\frac{1-H_1}{k})}\nonumber\\
&\quad\times e^{-\lambda_1(ah_1-h_1x_j)_+}\big(bh_2-h_2y_j\big)_+^{-(\frac{1}{2}+\frac{1-H_2}{k})}e^{-\lambda_2(bh_2-h_2y_j)_+}\Bigg)\nonumber\\
&=h_1h_2h_1^{-k(\frac{1}{2}+\frac{1-H_1}{k})}h_2^{-k(\frac{1}{2}+\frac{1-H_2}{k})}\int_{(\mathbb{R}^2)^ k}^\prime\, dW(h_1x_1, h_2y_1)\ldots dW(h_1x_k, h_2y_k)\nonumber\\
&\quad \times \Bigg(\int_0^{t} \,da \int_0^{s}\,db \prod_{j=1}^k \big(a-x_j\big)_+^{-(\frac{1}{2}+\frac{1-H_1}{k})}e^{-\lambda_1h_1(a-x_j)_+}\nonumber\\
&\quad\times\big(b-y_j\big)_+^{-(\frac{1}{2}+\frac{1-H_2}{k})} e^{-\lambda_2h_2(b-y_j)_+}\Bigg) \nonumber\\
&\overset{(d)}{=}h_1^{H_1}h_2^{H_2}\int_{(\mathbb{R}^2)^ k}^\prime\, dW(x_1, y_1)\ldots dW(x_k, y_k) \Bigg(\int_0^{t} \,da \int_0^{s}\,db\nonumber\\
&\quad\times \prod_{j=1}^k \big(a-x_j\big)_+^{-(\frac{1}{2}+\frac{1-H_1}{k})}e^{-\lambda_1h_1(a-x_j)_+}\big(b-y_j\big)_+^{-(\frac{1}{2}+\frac{1-H_2}{k})}e^{-\lambda_2h_2(b-y_j)_+}\Bigg)\label{lot}\\
&= h_1^{H_1}h_2^{H_2} Z^{k,\,H_1,H_2}_{h_1\lambda_1,h_2\lambda_2}(t,s),\nonumber
\end{align}
where in \eqref{lot} we have used the scaling property of the Brownian field.\\
From the definition of the two-parameter tempered Hermite process, one can see that for every $z_1, z_2 > 0$,
{\small\begin{align}
&Z^{k,\,H_1,H_2}_{\lambda_1,\lambda_2}(t+z_1,s+z_2)-Z^{k,\,H_1,H_2}_{\lambda_1,\lambda_2}(t,s+z_2)-Z^{k,\,H_1,H_2}_{\lambda_1,\lambda_2}(t+z_1,s)+Z^{k,\,H_1,H_2}_{\lambda_1,\lambda_2}(t,s)\nonumber\\
&\quad\overset{(d)}{=} \int_{(\mathbb{R}^2)^ k}^\prime\, dW(x_1, y_1)\ldots dW(x_k, y_k)\nonumber\\
&\quad \times \Bigg(\int_0^{t} \,da \int_0^{s}\,db \prod_{j=1}^k (a-x_j)_+^{-(\frac{1}{2}+\frac{1-H_1}{k})}e^{-\lambda_1(a-x_j)_+}(b-y_j)_+^{-(\frac{1}{2}+\frac{1-H_2}{k})}e^{-\lambda_2(b-y_j)_+}\Bigg)\nonumber\\
&\quad =Z^{k,\,H_1,H_2}_{\lambda_1,\lambda_2}(t,s).\nonumber
\end{align}}
\end{proof}
Now, we are going to study the continuity of the trajectories of the two-parameter tempered Hermite field. Firstly, let us recall the following two-parameter version of the Kolmogorov continuity theorem (see, e.g., \cite[Lemme 1]{Drap} and \cite[Theorem B.2]{book-Tudor}).
\begin{theorem}\label{kolll}
Let $(X(s,t))_{s,\,t\in T}$ be a two-parameter process, vanishing on the axis, with $T$ a compact subset of $\R$. Suppose that there exist constants $C,\, p >0$ and $x,\,y >1$ such that
$$\mathbb{E}\Big\vert X(t+z_1,s+z_2)-X(t,s+z_2)
-X(t+z_1,s)+X(t,s)\Big\vert^p\leq Cz_1^xz_2^y$$
for every $z_1,\,z_2>0$ and for every $s,\,t\in T$ such that $s+z_1,\,t+z_2\in T$. Then, $X$ admits a
continuous modification $\tilde{X}$. Moreover, $\tilde{X}$ has H\"{o}lder continuous paths of any orders
$x^\prime\in (0,\,\frac{x-1}{p})$, $y^\prime\in(0,\,\frac{y-1}{p})$ in the following sense: for every $\omega\in\Omega$, there exists a constant
$C_\omega> 0$ such that for every $s,\, t,\, s^\prime,\, t^\prime\in T$
$$\Big\vert X(s,t)(\omega)-X(s,t^\prime)(\omega)-X(s^\prime,t)(\omega)+X(s^\prime,t^\prime)(\omega)\Big\vert\leq C_\omega \vert t-t^\prime\vert\vert s-s^\prime\vert.$$
\end{theorem}
As a consequence of the previous results, we obtain the following proposition.
\begin{proposition}
The two-parameter tempered Hermite field $Z^{k,\,H_1,H_2}_{\lambda_1,\lambda_2}$ admits a version with continuous trajectories.
\end{proposition}
\begin{proof}
According to the proof of Lemma \ref{lemmma}, it is straightforward that
{\small\begin{align}
&\mathbb{E} \Big\vert Z^{k,\,H_1,H_2}_{\lambda_1,\lambda_2}(t+z_1,s+z_2)-Z^{k,\,H_1,H_2}_{\lambda_1,\lambda_2}(t,s+z_2)-Z^{k,\,H_1,H_2}_{\lambda_1,\lambda_2}(t+z_1,s)+
Z^{k,\,H_1,H_2}_{\lambda_1,\lambda_2}(t,s)\Big\vert^2\nonumber\\
&\leq \begin{cases} c_1 \vert z_1\vert ^{2H_1}\vert z_2\vert^{2H_2}& \frac{1}{2}<H_1,\,H_2<1,\\ c_2\vert z_1\vert ^{2}\vert z_2\vert^{2}& H_1,\,H_2>1,\end{cases}\nonumber
\end{align}}
where $c_1$ and $c_2$ are some positive constants. Using Theorem \ref{kolll} for $Z^{k,\,H_1,H_2}_{\lambda_1,\lambda_2}$ for $p=2$, $x=\min\{2H_1,2\}$, $y=\min\{2H_2,2\}$ and $c = \min\{c_1, c_2\}$, we get the desired result.
\end{proof}
Now, we are going to compute the covariance function  of the two-parameter tempered Hermite field.
\begin{proposition}The two-parameter tempered Hermite field $Z^{k,\,H_1,H_2}_{\lambda_1,\lambda_2}$ has the following covariance function:
{\small
\begin{align*}
\mathbb{E}\Big[Z^{k,\,H_1,H_2}_{\lambda_1,\lambda_2}(t,s)&Z^{k,\,H_1,H_2}_{\lambda_1,\lambda_2}(u,v)\Big]
=\Bigg[\frac{ \Gamma(\frac{1}{2}-\frac{1-H_1}{k})}{\sqrt{\pi}(2\lambda_1)^{\frac{H_1-1}{k}}}\Bigg]^k\Bigg[\frac{ \Gamma(\frac{1}{2}-\frac{1-H_2}{k})}{\sqrt{\pi}(2\lambda_2)^{\frac{H_2-1}{k}}}\Bigg]^k \\
   &\times \int_0^{t}\int_0^{s}\Big[ \vert u_1-v_1\vert^{\frac{H_1-1}{k}} K_{\frac{H_1-1}{k}}\big(\lambda_1\vert u_1-v_1\vert \big) \Big]^k \, du_1 dv_1\\
  &\times \int_0^{u}\int_0^{v}\Big[ \vert u_2-v_2\vert^{\frac{H_2-1}{k}} K_{\frac{H_2-1}{k}}\big(\lambda_2\vert u_2-v_2\vert \big) \Big]^k \, du_2 dv_2.
\end{align*}}
\end{proposition}
\begin{proof} By applying the Fubini theorem and the isometry of multiple Wiener-It\^{o} integrals we have
 {\small
\begin{align*}
&\mathbb{E}\Big[Z^{k,\,H_1,H_2}_{\lambda_1,\lambda_2}(t,s)Z^{k,\,H_1,H_2}_{\lambda_1,\lambda_2}(u,v)\Big]=\mathbb{E}\Bigg[\Bigg\{\int\limits_{(\mathbb{R}^2)^ k}^\prime \vspace{-1em}dW(x_1, y_1)\ldots dW(x_k, y_k)\nonumber\\
&\qquad \times \Bigg(\int\limits_0^t da \int\limits_0^s db \prod_{j=1}^k (a-x_j)_+^{-(\frac{1}{2}+\frac{1-H_1}{k})}e^{-\lambda_1(a-x_j)_+}(b-y_j)_+^{-(\frac{1}{2}+\frac{1-H_2}{k})}e^{-\lambda_2(b-y_j)_+}\Bigg)\Bigg\}\\
&\qquad\times\Bigg\{ \int_{(\mathbb{R}^2)^ k}^\prime\, dW(x_1, y_1)\ldots dW(x_k, y_k)\Bigg(\int_0^u \,da' \int_0^v\,db' \nonumber\\
&\qquad \times \prod_{j=1}^k (a'-x_j)_+^{-(\frac{1}{2}+\frac{1-H_1}{k})}e^{-\lambda_1(a'-x_j)_+}(b'-y_j)_+^{-(\frac{1}{2}+\frac{1-H_2}{k})}e^{-\lambda_2(b'-y_j)_+}\Bigg)\Bigg\}\Bigg]\\
&\quad= k!\int_{(\mathbb{R}^2)^ k}^\prime \,dx_1\ldots dx_kdy_1\ldots dy_k\\
&\qquad \times \Bigg(\int_0^t \,da \int_0^s\,db \prod_{j=1}^k (a-x_j)_+^{-(\frac{1}{2}+\frac{1-H_1}{k})}e^{-\lambda_1(a-x_j)_+}(b-y_j)_+^{-(\frac{1}{2}+\frac{1-H_2}{k})}e^{-\lambda_2(b-y_j)_+}\Bigg)\\
&\qquad \times \Bigg(\int_0^u \,da' \int_0^v\,db' \prod_{j=1}^k (a'-x_j)_+^{-(\frac{1}{2}+\frac{1-H_1}{k})}e^{-\lambda_1(a'-x_j)_+}(b'-y_j)_+^{-(\frac{1}{2}+\frac{1-H_2}{k})}e^{-\lambda_2(b'-y_j)_+}\Bigg)\\
&\quad = k! \int_0^t \,da \int_0^s\,db \int_0^u \,da' \int_0^v\,db'\\
&\qquad \times  \Bigg[\int_\R (a-x)_+^{-(\frac{1}{2}+\frac{1-H_1}{k})}(a'-x)_+^{-(\frac{1}{2}+\frac{1-H_1}{k})}e^{-\lambda_1(a-x)_+}e^{-\lambda_1(a'-x)_+}\,dx\Bigg]^k\\
&\qquad \times \Bigg[\int_\R  (b-y)_+^{-(\frac{1}{2}+\frac{1-H_2}{k})}(b'-y)_+^{-(\frac{1}{2}+\frac{1-H_2}{k})}e^{-\lambda_2(b-y)_+}e^{-\lambda_2(b'-y)_+}\,dy\Bigg]^k\\
&\quad =k! \int_0^t \,da \int_0^s\,db \int_0^u \,da' \int_0^v\,db'\\
&\qquad \times  \Bigg[\int\limits_{-\infty}^{\min(a,a')} (a-x)^{-(\frac{1}{2}+\frac{1-H_1}{k})}(a'-x)^{-(\frac{1}{2}+\frac{1-H_1}{k})}e^{-\lambda_1(a-x)}e^{-\lambda_1(a'-x)}\,dx\Bigg]^k\\
&\qquad \times \Bigg[\int\limits_{-\infty}^{\min(b,b')}  (b-y)^{-(\frac{1}{2}+\frac{1-H_2}{k})}(b'-y)^{-(\frac{1}{2}+\frac{1-H_2}{k})}e^{-\lambda_2(b-y)}e^{-\lambda_2(b'-y)}\,dy\Bigg]^k.
\intertext{Finally, we get}
&\mathbb{E}\Big[Z^{k,\,H_1,H_2}_{\lambda_1,\lambda_2}(t,s)Z^{k,\,H_1,H_2}_{\lambda_1,\lambda_2}(u,v)\Big]=k! \int\limits_0^t da \int\limits_0^s\,db \int\limits_0^u \,da' \int_0^v\,db'\\
&\qquad \times  \Bigg[\int\limits_{0}^{+\infty} \xi^{-(\frac{1}{2}+\frac{1-H_1}{k})}(|a-a'|+\xi)^{-(\frac{1}{2}+\frac{1-H_1}{k})}e^{-\lambda_1\xi}e^{-\lambda_1(|a-a'|+\xi)}\,d\xi\Bigg]^k\\
&\qquad\times \Bigg[\int\limits_{0}^{+\infty}  \omega^{-(\frac{1}{2}+\frac{1-H_2}{k})}(|b-b'|+\omega)^{-(\frac{1}{2}+\frac{1-H_2}{k})}e^{-\lambda_2\omega}e^{-\lambda_2(|b-b'|+\omega)}\,d\omega\Bigg]^k\\
&\quad= k! \int_0^t \,da \int_0^s\,db \int_0^u \,da' \int_0^v\,db' e^{-\lambda_1 k\vert a-a'\vert} \vert a-a'\vert^{2(H_1-1)} e^{-\lambda_2 k\vert b-b'\vert} \vert b-b'\vert^{2(H_2-1)}\\
&\qquad\times \Bigg[ \int\limits_0^{+\infty} x^{-(\frac{1}{2}+\frac{1-H_1}{k})} (x+1)^{-(\frac{1}{2}+\frac{1-H_1}{k})} e^{-2\lambda_1 \vert a-a'\vert x}\, dx\Bigg]^k\\
&\qquad\times \Bigg[ \int\limits_0^{+\infty} y^{-(\frac{1}{2}+\frac{1-H_2}{k})} (y+1)^{-(\frac{1}{2}+\frac{1-H_2}{k})} e^{-2\lambda_2 \vert b-b'\vert y}\, dy\Bigg]^k\\
&\quad = k! \int_0^t \,da \int_0^s\,db \int_0^u \,da' \int_0^v\,db' e^{-\lambda_1 k\vert a-a'\vert} \vert a-a'\vert^{2(H_1-1)} e^{-\lambda_2 k\vert b-b'\vert} \vert b-b'\vert^{2(H_2-1)}\\
&\qquad \times \Bigg[ \frac{\Gamma(\frac{1}{2}-\frac{1-H_1}{k})}{\sqrt{\pi}}\Big(\frac{1}{2\lambda_1\vert a-a'\vert}\Big)^{\frac{H_1-1}{k}} e^{\lambda_1\vert a-a'\vert} K_{\frac{H_1-1}{k}}\big(\lambda_1\vert a-a'\vert \big) \Bigg]^k\\
&\qquad \times \Bigg[ \frac{\Gamma(\frac{1}{2}-\frac{1-H_2}{k})}{\sqrt{\pi}}\Big(\frac{1}{2\lambda_2\vert b-b'\vert}\Big)^{\frac{H_2-1}{k}} e^{\lambda_2\vert b-b'\vert} K_{\frac{H_2-1}{k}}\big(\lambda_2\vert b-b'\vert \big) \Bigg]^k\\
&\quad= k! \int_0^t \,da \int_0^s\,db \int_0^u \,da' \int_0^v\,db'\\
&\qquad \times \Bigg[ \frac{\Gamma(\frac{1}{2}-\frac{1-H_1}{k})}{\sqrt{\pi} (2\lambda_1)^{\frac{H_1-1}{k}}}\Bigg]^k \times \Big[ \vert a-a'\vert^{\frac{H_1-1}{k}} K_{\frac{H_1-1}{k}}\big(\lambda_1\vert a-a'\vert \big) \Big]^k\\
&\qquad \times \Bigg[ \frac{\Gamma(\frac{1}{2}-\frac{1-H_2}{k})}{\sqrt{\pi} (2\lambda_2)^{\frac{H_2-1}{k}}}\Bigg]^k \times \Big[ \vert b-b'\vert^{\frac{H_2-1}{k}} K_{\frac{H_2-1}{k}}\big(\lambda_2\vert b-b'\vert \big) \Big]^k.\\
\end{align*}
}
Thus,
\begin{align*}
&\mathbb{E}\Big[Z^{k,\,H_1,H_2}_{\lambda_1,\lambda_2}(t,s)Z^{k,\,H_1,H_2}_{\lambda_1,\lambda_2}(u,v)\Big]\\
&\quad =\Bigg[\frac{ \Gamma(\frac{1}{2}-\frac{1-H_1}{k})}{\sqrt{\pi}(2\lambda_1)^{\frac{H_1-1}{k}}}\Bigg]^k\Bigg[\frac{ \Gamma(\frac{1}{2}-\frac{1-H_2}{k})}{\sqrt{\pi}(2\lambda_2)^{\frac{H_2-1}{k}}}\Bigg]^k \\
&\hspace{3em}\times\int\limits_0^{t}\int\limits_0^{s}\Big[ \vert a-a'\vert^{\frac{H_1-1}{k}} K_{\frac{H_1-1}{k}}\big(\lambda_1\vert a-a'\vert \big) \Big]^k \, da da'\\
&\hspace{3em} \times \int\limits_0^{u}\int\limits_0^{v}\Big[ \vert b-b'\vert^{\frac{H_2-1}{k}} K_{\frac{H_2-1}{k}}\big(\lambda_2\vert b-b'\vert \big) \Big]^k \, db db',
\end{align*}
which finishes the proof.
\end{proof}
\begin{remark} From the previous proposition we can see that the covariance function of the two-parameter tempered Hermite field varies with respect to $k\geq1$ contrary to the Hermite field (see \cite{HS3}) which has the same covariance structure for all $k\geq 1$ (the latter coincides with the covariance of the fractional Brownian field).\end{remark}
\section{Spectral representation of the two-parameter tempered Hermite field}
\label{sect4}
The aim of this section is to analyze more deeply the class of the two-parameter tempered Hermite field. The representation \eqref{tHs} is defined on the real line and on the time domain. In the sequel, we will introduce equivalent spectral integral representations defined on the real line of the process.

It should be remembered that for one parameter processes, the tempered Hermite process $Z^{k,\,H}_\lambda$ ($H>1/2$ and $\lambda>0$) has the following spectral domain representation  (see, e.g., \cite[Theorme 1.1]{taqq} and \cite[Theorem 6.3]{Taqqu79}):
$$Z_{\lambda,H}^k(t) = C_{H,k} \int_{\mathbb{R}^k}^{\prime \prime} \frac{e^{it(\omega_1+\ldots +\omega_k)}-1}{i(\omega_1+\ldots+\omega_k)} \times \prod_{j=1}^k (\lambda+i\omega)^{-(\frac{1}{2}-\frac{1-H}{k})} \widehat{B}(d\omega_1)\ldots \widehat{B}(d\omega_k),$$
where $\widehat{B}$ is a suitable complex-valued Gaussian random measure on $\big(\,\mathbb{R},\mathcal{B}(\mathbb{R})\,\big)$ and the double prime on the integral indicates that one does not integrate on diagonals where $\omega_i=\omega_j$, $i\neq j$.\\
We begin this section by defining the Hermitian random measures on $(\mathbb{R}^2, \mathcal{B}(\mathbb{R}^2))$ and the corresponding Wiener integral with respect to it in Subsection \ref{sub:5.1}. Next, we give the spectral representations theorem for two-parameter stochastic processes in Subsection \ref{sub:5.2}. Finally, we study the case of the two-tempered Hermite sheet in Subsection \ref{sub:5.3}.
\subsection{Hermitian random measures in $\big(\mathbb{R}^2, \mathcal{B}(\mathbb{R}^2)\big)$}
\label{sub:5.1}
\begin{definition}
Let $m$ be a symmetric random measure on $\big(\mathbb{R}^2, \mathcal{B}(\mathbb{R}^2)\big)$ in the sense that
\begin{equation}\label{loujayn1}
  m(A\times B)=m(-(A\times B)),\quad \text{for $A\times B\in \mathcal{B}(\mathbb{R}^2)$},
\end{equation}
where
\begin{equation*}
  -(A\times B)\,=\, \big\{(x,y)\in\mathbb{R}^2\,\,:\,\,(-x,-y)\in A\times B\big\}.
\end{equation*}
A Hermitian (complex) Gaussian random measure  on $\big(\mathbb{R}^2, \mathcal{B}(\mathbb{R}^2)\big)$ with a symmetric control measure $m$ is a collection of complex-valued random fields $\big\{\, \widehat{W}(A\times B)\,; \,\, A\times B \in \mathcal{B}(\mathbb{R}^2)_0\,\big\}$ defined on some probability space $\big(\Omega,\mathcal{F},\mathbb{P}\big)$ such that
\begin{equation}\label{loujayn2}
  \overline{\widehat{W}(A\times B)}\,=\,  \widehat{W}(-(A\times B));\quad A\times B\in \mathcal{B}(\mathbb{R}^2)_0,
\end{equation}
 where
 \begin{equation*}
   \mathcal{B}(\mathbb{R}^2)_0\,=\, \big\{\, A\times B\in \mathcal{B}(\mathbb{R}^d)\,:\,\,\, m(A\times B)<\infty \,\big\}.
 \end{equation*}
\end{definition}
We note that relations \eqref{loujayn1} and \eqref{loujayn2} are often written as $m(dxdy)= m(-dxdy)$ and $\overline{\widehat{W}(dxdy)}= \widehat{W}(-dxdy)$, respectively. There is several properties of Hermitian (complex) Gaussian random measure defined on $\big(\mathbb{R}^2, \mathcal{B}(\mathbb{R}^2)\big)$ that can be found for example in \cite[Appendix B]{Taqqu17} and \cite[Chapter 9]{hh}.
In the following statements, we suppose that the sets belong to $\mathcal{B}(\mathbb{R}^2)_0$. We have:
\begin{enumerate}
  \item $\mathbb{E}\big[\, \widehat{W}(A\times B)\,\big]=0$ and $\mathbb{E}\big[\, \widehat{W}(A_1\times B_1)\overline{\widehat{W}(A_2\times B_2)} \,\big]\,=\, m((A_1\times B_1)\cap (A_2\times B_2))$.
  \item If $A\times B\cap (-(A\times B))= \emptyset$, then  $\mathbb{E}\big[\, \widehat{W}(A\times B)^2\,\big]=0$.
  \item $\text{Re} \widehat{W}(A\times B)$ and $\text{Im}\widehat{W}(A\times B)$ are independent.
  \item If $A_1\times B_1\cup(-(A_1\times B_1)),\ldots,A_n\times B_n\cup(-(A_n\times B_n))$ are disjoint, then $\widehat{W}(A_1\times B_1),\ldots,\widehat{W}(A_n\times B_n)$ are independent.
\end{enumerate}

Having defined a Hermitian Gaussian random measures $\widehat{W}$, we shall now define, $I^{\widehat{W}}_k$, the multiple Wiener integrals with respect to $\widehat{W}$. To define such stochastic integrals, one firstly introduces $\mathcal{H}_2^k$, the real Hilbert space of complex-valued functions $f((x_1,y_1), \ldots,(x_k,y_k))$, $(x_i,y_i) \in\mathbb{R}^2$, $i=1,3,\ldots, k$ that are even, i.e. $f((x_1,y_1), \ldots,(x_k,y_k))= \overline{f(-(x_1,y_1), \ldots,-(x_k,y_k))}$ and square integrable, that is,
$$\|f\|^2 \,=\, \int_{(\mathbb{R}^2)^k} \vert f((x_1,y_1), \ldots,(x_k,y_k)) \vert^2 dx_1dy_1\ldots dx_kdy_k< \infty.$$
The inner product is similarly defined for $f,\,g\in \mathcal{H}_2^k$ by
$$\langle f,g\rangle _{\mathcal{H}_2^k} \,=\, \int f((x_1,y_1), \ldots,(x_k,y_k)) \overline{g((x_1,y_1), \ldots,(x_k,y_k)) } dx_1dy_1\ldots dx_kdy_k.$$
The integrals $I^{\widehat{W}}_k$ are then defined through an isometric mapping from $\mathcal{H}_2^k$ to $L^2(\Omega)$ :
\begin{equation*}
 f\mapsto I^{\widehat{W}}_k(f) \,=\, \int_{(\mathbb{R}^2)^k}^{\prime\prime} f((x_1,y_1), \ldots,(x_k,y_k))\, \widehat{W}(dx_1dy_1)\ldots \widehat{W}(dx_kdy_k).
\end{equation*}
The mapping is defined in such a way, that heuristically, one disregards integration over hyperplanes. The fact that both $f$ and $\widehat{W}$ are even ensuring that $I^{\widehat{W}}_k(f)$ is a real-valued random field.
\subsection{Spectral representations of two-parameter stochastic processes}
\label{sub:5.2}
In this section, we are interested in the relation between  the classical multiple Wiener-It\^{o} integrals with respect to the standard Brownian field $I_k^{W}$ defined on Section \ref{sect2} and the one defined with respect to the random spectral measure $I^{\widehat{W}}_k$. According to \cite[Lemma 6.1 and Remark 6.2]{Taqqu79}, we have the following result:
\begin{proposition}\label{amam}
Let $A((\xi_1,\omega_1),\ldots,(\xi_k,\omega_j))$ be a real-valued symmetric function in $L^2\big((\mathbb{R}^2)^k\big)$ and let
\begin{align}
&\quad \mathcal{F} [A]((\lambda_1, \zeta_1),\ldots,(\lambda_k,\zeta_k))\nonumber \\
&= \frac{1}{(2\pi)^k}\int_{(\mathbb{R}^2)^k}  e^{i\sum_{j=1}^k \xi_j\lambda_j}e^{i\sum_{j=1}^k \omega_j\zeta_j}A((\xi_1,\omega_1),\ldots,(\xi_k,\omega_j)) \,d\xi_1d\omega_1\ldots d\xi_kd\omega_k,
\end{align}
be its Fourier transform. Then,
\begin{align}
&\quad \int_{(\mathbb{R}^2)^k}^\prime  A((\xi_1,\omega_1),\ldots,(\xi_k,\omega_k)) \,dW (\xi_1,\omega_1)\ldots dW (\xi_k,\omega_k)\nonumber \\ &\overset{(d)}{=} \int_{(\mathbb{R}^2)^k}^{\prime\prime} \mathcal{F} A((\lambda_1 \zeta_1),\ldots,(\lambda_k\zeta_k)) \widehat{W}(d\lambda_1,d\zeta_1)\ldots \widehat{W}(d\lambda_k,d\zeta_k).
\end{align}
\end{proposition}

\subsection{The case of the two-parameter tempered Hermite fields}
\label{sub:5.3}
\begin{proposition}
Let $H_1,\,H_2>\frac{1}{2}$ and $\lambda_1,\,\lambda_2>0$. The two-parameter tempered Hermite random field given by \eqref{tHs} has the following spectral domain representation
\begin{align*}
  &Z^{k,\,H_1,H_2}_{\lambda_1,\lambda_2}(s,t) = C_{H_1,H_2,k}\int_{(\mathbb{R}^2)^k}^{\prime\prime} -\frac{(e^{it\sum_{j=1}^k\xi_j}-1)(e^{is\sum_{j=1}^k\omega_j}-1)}{\sum_{j=1}^k \xi_j \sum_{j=1}^k\omega_j}\nonumber\\
   &\quad \times \prod_{j=1}^k(\lambda_1+i\xi_j)^{-(\frac{1}{2}-\frac{1-H_1}{k})}(\lambda_2+i\omega_j)^{-(\frac{1}{2}-\frac{1-H_2}{k})}\,
   \widehat{W}(d\xi_1d\omega_1)\ldots \widehat{W}(d\xi_kd\omega_k), \nonumber
\end{align*}
where $\widehat{W}(.)$ is a complex-valued Gaussian random measure on $\big(\mathbb{R}^2, \mathcal{B}(\mathbb{R}^2)\big)$, and $$C_{H_1,H_2,k}=  \Bigg[\frac{\Gamma\Big(\frac{1}{2}-\frac{1-H_1}{k}\Big)\Gamma\Big(\frac{1}{2}-\frac{1-H_2}{k}\Big)}{2\pi}\Bigg]^k$$
is a constant depending on $H_1,\,H_2$, and $k$. The double prime $^{\prime\prime}$ on the integral indicates that one does not integrate on the hyperplanes $(\xi_{j_1},\omega_{j_1})=(\xi_{j_2},\omega_{j_2})$, $j_1\neq j_2$.
\end{proposition}
\begin{proof}
Let $h_{s,t}^{H_1,H_2,\lambda_1,\lambda_2}: (\mathbb{R}^2)^k \to \R $ the function defined in Lemma \ref{lemmma}:
\begin{align*}
&\quad h_{s,t}^{H_1,H_2,\lambda_1,\lambda_2}((x_1,y_1),\dots,(x_k,y_k))\\
 &= \int_0^t\int_0^s\prod_{j=1}^k (a-x_j)_+^{-(\frac{1}{2}+\frac{1-H_1}{k})}e^{-\lambda_1(a-x_j)_+}(b-y_j)_+^{-(\frac{1}{2}+\frac{1-H_2}{k})}e^{-\lambda_2(b-y_j)_+} \,da\,db.
\end{align*}
Let us first compute its Fourier transform.
{\small\begin{align*}
 & \mathcal{F}[ h_{s,t}^{H_1,H_2,\lambda_1,\lambda_2}]((\xi_1,\omega_1),\dots,(\xi_k,\omega_k)) \nonumber \\
 &\quad = \frac{1}{(2\pi)^k}\int_{(\mathbb{R}^2)^k}  e^{i\sum_{j=1}^k \xi_jx_j}e^{i\sum_{j=1}^k \omega_jy_j} h_{s,t}^{H_1,H_2,\lambda_1,\lambda_2}((x_1,y_1),\ldots,(x_k,y_k)) \,dx_1dy_1\ldots dx_kdy_k\\
 &\quad = \frac{1}{(2\pi)^k}\int_{(\mathbb{R}^2)^k} \,dx_1dy_1\ldots dx_kdy_k e^{i\sum_{j=1}^k \xi_jx_j}e^{i\sum_{j=1}^k \omega_jy_j}\\
 &\qquad \times \int_0^t\int_0^s\prod_{j=1}^k (a-x_j)_+^{-(\frac{1}{2}+\frac{1-H_1}{k})}e^{-\lambda_1(a-x_j)_+}(b-y_j)_+^{-(\frac{1}{2}+\frac{1-H_2}{k})}e^{-\lambda_2(b-y_j)_+} \,da\,db\\
 &\quad = \frac{1}{(2\pi)^k} \int_0^t\int_0^s \Big[\int_{\R^k}e^{i\sum_{j=1}^k \xi_jx_j}\prod_{j=1}^k (a-x_j)_+^{-(\frac{1}{2}+\frac{1-H_1}{k})}e^{-\lambda_1(a-x_j)_+}\,dx_1\ldots dx_k\Big]\, da\,db\\
 &\quad = \frac{1}{(2\pi)^k} \int_0^t\int_0^s \Big[(-1)^k\int_{\R^k}e^{i\sum_{j=1}^k \xi_j(a-X_j)}\prod_{j=1}^k (X_j)_+^{-(\frac{1}{2}+\frac{1-H_1}{k})}e^{-\lambda_1(X_j)_+}\,dX_1\ldots dX_k\Big]\\
 &\qquad \times \Big[(-1)^k\int_{\R^k}e^{i\sum_{j=1}^k \omega_j(b-Y_j)}\prod_{j=1}^k (Y_j)_+^{-(\frac{1}{2}+\frac{1-H_2}{k})}e^{-\lambda_2(Y_j)_+}\,dY_1\ldots dY_k\Big]\, da\,db.
  \end{align*}}
  Then,
\begin{align*}
 & \mathcal{F}[ h_{s,t}^{H_1,H_2,\lambda_1,\lambda_2}]((\xi_1,\omega_1),\dots,(\xi_k,\omega_k)) \nonumber \\
 &\quad = \frac{1}{(2\pi)^k} \int_0^t e^{i\sum_{j=1}^k \xi_ja}da \int_0^s e^{i\sum_{j=1}^k \omega_jb}db\\
 &\qquad \times  \Big[\int_{\R^k}\prod_{j=1}^k (X_j)_+^{-(\frac{1}{2}+\frac{1-H_1}{k})}e^{-(\lambda_1+i\xi_j)(X_j)_+}\,dX_1\ldots dX_k\Big]\\
 &\qquad \times \Big[\int_{\R^k}\prod_{j=1}^k (Y_j)_+^{-(\frac{1}{2}+\frac{1-H_2}{k})}e^{-(\lambda_2+i\omega_j)(Y_j)_+}\,dY_1\ldots dY_k\Big]\\
 & =\frac{-1}{(2\pi)^k}\frac{(e^{it\sum_{j=1}^k\xi_j}-1)(e^{is\sum_{j=1}^k\omega_j}-1)}{\sum_{j=1}^k \xi_j \sum_{j=1}^k\omega_j}\Gamma\big(\frac{1}{2}-\frac{1-H_1}{k}\big)^k\prod_{j=1}^k(\lambda_1+i\xi_j)^{-(\frac{1}{2}-\frac{1-H_1}{k})}\\
&\quad\times \Gamma\big(\frac{1}{2}-\frac{1-H_2}{k}\big)^k\prod_{j=1}^k(\lambda_2+i\omega_j)^{-(\frac{1}{2}-\frac{1-H_2}{k})}.
\end{align*}
Let $\widehat{W}(.)$ be a complex-valued Gaussian random measure on $\big(\mathbb{R}^2, \mathcal{B}(\mathbb{R}^2)\big)$. Using Proposition \ref{amam}, we get
\begin{align*}
  &Z^{k,\,H_1,H_2}_{\lambda_1,\lambda_2}(s,t) \overset{(d)}{=} C_{H_1,H_2,k}\int_{(\mathbb{R}^2)^k}^{\prime\prime} -\frac{(e^{it\sum_{j=1}^k\xi_j}-1)(e^{is\sum_{j=1}^k\omega_j}-1)}{\sum_{j=1}^k \xi_j \sum_{j=1}^k\omega_j}\nonumber\\
   & \hspace{3em}\times \prod_{j=1}^k(\lambda_1+i\xi_j)^{-(\frac{1}{2}-\frac{1-H_1}{k})}(\lambda_2+i\omega_j)^{-(\frac{1}{2}-\frac{1-H_2}{k})}\,
   \widehat{W}(d\xi_1d\omega_1)\ldots \widehat{W}(d\xi_kd\omega_k). \nonumber
\end{align*}
\end{proof}
\begin{remark}
In \cite[Remark 2]{HS3}, the authors said that it will be interesting to find the spectral domain representation of the Hermite field \eqref{H}. So, taking $\lambda_1=\lambda_2=0$ and using the previous results, one can write
\begin{eqnarray}
  Z^{k,\,H_1,H_2}(s,t) &\overset{(d)}{=}& C_{H_1,H_2,k}\int_{(\mathbb{R}^2)^k}^{\prime\prime} -\frac{(e^{it\sum_{j=1}^k\xi_j}-1)(e^{is\sum_{j=1}^k\omega_j}-1)}{\sum_{j=1}^k \xi_j \sum_{j=1}^k\omega_j}\nonumber\\
   && \times \prod_{j=1}^k(i\xi_j)^{-(\frac{1}{2}-\frac{1-H_1}{k})}(i\omega_j)^{-(\frac{1}{2}-\frac{1-H_2}{k})}\,
   \widehat{W}(d\xi_1d\omega_1)\ldots \widehat{W}(d\xi_kd\omega_k). \nonumber
\end{eqnarray}
where $H_1,\,H_2\in (\frac{1}{2},1)$.
\end{remark}


\section{Two-parameter tempered fractional calculus}
\label{sect5}
\subsection{Two-parameter tempered fractional integrals}
In this subsection, we will give the definitions of multiple tempered fractional integrals and derive their properties.
\begin{definition}[Two-parameter tempered fractional integrals]
Let $\alpha_1,\,\alpha_d,\,\lambda_1,\,\lambda_2>0$. We denote $\alpha=(\alpha_1,\alpha_2)$ and $\lambda=(\lambda_1,\lambda_2)$. Let $f$ be a function belonging to  $L^p(\mathbb{R}^2)$ (where $1\leq p<\infty$). The left and
the right two-parameter tempered fractional integrals of order $\alpha$ are, respectively, defined as
\begin{align*}
\mathbb{I}^{\alpha,\,\lambda}_+(f(t,s)) &= \frac{1}{\Gamma(\alpha_1)\Gamma(\alpha_2)}\int_{\R^2}e^{-\lambda_1(t-u)_+}(t-u)_+^{\alpha_1-1}\\
\nonumber&\quad\times e^{-\lambda_2(s-v)_+}(s-v)_+^{\alpha_2-1}f(u,v)\,dudv
\intertext{and}
\mathbb{I}^{\alpha,\,\lambda}_-(f(t,s)) &= \frac{1}{\Gamma(\alpha_1)\Gamma(\alpha_2)}\int_{\R^2}e^{\lambda_1(u-t)_+}(u-t)_+^{\alpha_1-1}\\
\nonumber&\quad\times e^{\lambda_2(v-s)_+}(v-s)_+^{\alpha_2-1}f(u,v)\,dudv
\end{align*}
where $\displaystyle\Gamma(\alpha_i)=\int_0^{+\infty}e^{-x}x^{\alpha_i-1}\,dx$ is the Euler gamma function, and $(x)_+ = xI(x > 0)$.
\end{definition}
When $\lambda_1,\,\lambda_2=0$, these definitions reduce to the (positive and negative) multiple Riemann-Liouville fractional integrals, which extend the usual operations of multiple iterated integration to a multiple fractional order.
The following results gather some basic properties of fractional integrals
\begin{proposition}
For any $\alpha_1,\,\alpha_2>0$, $\lambda_1,\,\lambda_2>0$, and $p\geq 1\in\mathbb{N}$, the multiple parameters tempered fractional integrals $ \mathbb{I}^{\alpha,\,\lambda}_+$, and $\mathbb{I}^{\alpha,\,\lambda}_-$ have the following
properties:
\begin{enumerate}[(i)]
\item{Reflection property:} If $Q$ is the reflection operator defined by $(Q f)(u,v) = f(-u,-v)$, then
$$Q \mathbb{I}^{\alpha,\,\lambda}_\pm f= \mathbb{I}^{\alpha,\,\lambda}_\pm Qf.$$
\item{Semigroup property:} For $f\in L^1(\R^2)$ we have
$$\mathbb{I}^{\alpha,\,\lambda}_\pm \mathbb{I}^{\beta,\,\lambda}_\pm f\,=\, \mathbb{I}^{\alpha+\beta,\,\lambda}_\pm f,\quad \alpha=(\alpha_1,\alpha_2),\,\beta=(\beta_1,\beta_2)>(0,0).$$
\item{Two-parameter tempered fractional integration by parts formula:} Suppose $f,\,g\in L^2(\R^2)$. Then
$$\int_{\R^2}f(x,y)\mathbb{I}^{\alpha,\,\lambda}_+g(x,y)\,dxdy\,=\, \int_{\R^2}\mathbb{I}^{\alpha,\,\lambda}_-f(x,y)g(x,y)\,dxdy.$$
\end{enumerate}
\end{proposition}
\begin{proof}
The property $(i)$ is elementary. In fact,
\begin{eqnarray*}
  Q \mathbb{I}^{\alpha,\,\lambda}_\pm f(u,v) &=& \mathbb{I}^{\alpha,\,\lambda}_\pm f(-u,-v) \\
   &=&  \mathbb{I}^{\alpha,\,\lambda}_\pm Qf(u,v).
\end{eqnarray*}
The proof of $(ii)$ is direct.
\begin{align*}
&\big(\mathbb{I}^{\alpha,\,\lambda}_+ \mathbb{I}^{\beta,\,\lambda}_+\big) f(t,s)\\
&=\frac{1}{\Gamma(\alpha_1)\Gamma(\alpha_2)\Gamma(\beta_1)\Gamma(\beta_2)}\int_{\R^2}
e^{-\lambda_1(t-u)_+}(t-u)_+^{\alpha_1-1}e^{-\lambda_2(s-v)_+}(s-v)_+^{\alpha_2-1}\\
&\quad\times \int_{\R^2}e^{-\lambda_1(u-x)_+}(u-x)_+^{\beta_1-1}e^{-\lambda_2(v-y)_+}(v-y)_+^{\beta_2-1}f(x,y)\,dxdy\,dudv.
\end{align*}
By changing the order of integration using Fubini's theorem and making the change of variables $u=x+(t-x)\omega_1$ and $v=y+(s-y)\omega_2$:
{\footnotesize\begin{align*}
&\big(\mathbb{I}^{\alpha,\,\lambda}_+ \mathbb{I}^{\beta,\,\lambda}_+\big) f(t,s)\\
&=\frac{1}{\Gamma(\alpha_1)\Gamma(\alpha_2)\Gamma(\beta_1)\Gamma(\beta_2)}\int_{\R^2}f(x,y)\Bigg[\int_{\R^2}
e^{-\lambda_1(t-u)_+}(t-u)_+^{\alpha_1-1}e^{-\lambda_2(s-v)_+}(s-v)_+^{\alpha_2-1}\\
&\quad\times e^{-\lambda_1(u-x)_+}(u-x)_+^{\beta_1-1}e^{-\lambda_2(v-y)_+}(v-y)_+^{\beta_2-1}\,dudv\Bigg]\,dxdy\\
&=\frac{B(\alpha_1,\beta_1)B(\alpha_2,\beta_2)}{\Gamma(\alpha_1)\Gamma(\alpha_2)\Gamma(\beta_1)\Gamma(\beta_2)}\int_{\R^2}
f(x,y)e^{-\lambda_1(t-x)_+}(t-x)_+^{\alpha_1+\beta_1-1}e^{-\lambda_2(s-y)_+}(s-y)_+^{\alpha_2+\beta_2-1}\,dxdy\\
&=\mathbb{I}^{\alpha+\beta,\,\lambda}_+ f (t,s).
\end{align*}}
Next, we prove similarly that $\mathbb{I}^{\alpha,\,\lambda}_- \mathbb{I}^{\beta,\,\lambda}_- f\,=\, \mathbb{I}^{\alpha+\beta,\,\lambda}_- f$.\\
The property $(iii)$ would follow immediately if one could change the order of integration in:
{\footnotesize
\begin{align*}
&\int_{\R^2}f(x,y)\mathbb{I}^{\alpha,\,\lambda}_+g(x,y)\,dxdy\\
&\quad =\int_{\R^2}f(x,y) \frac{1}{\Gamma(\alpha_1)\Gamma(\alpha_2)}\int_{\R^2}e^{-\lambda_1(x-u)_+}(x-u)_+^{\alpha_1-1}e^{-\lambda_2(y-v)_+}(y-v)_+^{\alpha_2-1}g(u,v)\,dudvdxdy\\
&\quad =\int_{\R^2}f(x,y) \frac{1}{\Gamma(\alpha_1)\Gamma(\alpha_2)}\int_{-\infty}^x\int_{-\infty}^y e^{-\lambda_1(x-u)}(x-u)^{\alpha_1-1}e^{-\lambda_2(y-v)}(y-v)^{\alpha_2-1}g(u,v)\,dudvdxdy\\
&\quad =\int_{\R^2}\frac{g(u,v)}{\Gamma(\alpha_1)\Gamma(\alpha_2)}\int_u^{+\infty}\int_v^{+\infty} f(x,y)e^{-\lambda_1(x-u)}(x-u)^{\alpha_1-1}e^{-\lambda_2(y-v)}(y-v)^{\alpha_2-1} dxdydudv\\
&\quad = \int_{\R^2}\mathbb{I}^{\alpha,\,\lambda}_-f(x,y)g(x,y)\,dx\,dy
\end{align*}}
and this completes the proof.
\end{proof}
Now, we will derive other properties of the two-parameter tempered fractional integrals that will be needed in the rest of this paper.
\begin{lemma}\label{vvv1}
For any $\alpha=(\alpha_1,\alpha_2)>(0,0)$, $\lambda=(\lambda_1,\lambda_2)>(0,0)$, and $1\leq p<\infty$, $ \mathbb{I}^{\alpha,\,\lambda}_\pm$ is a bounded linear operator on $L^p(\R^2)$ such that
\begin{equation}\label{Yo}
  \big\|\mathbb{I}^{\alpha,\,\lambda}_\pm f\big\|_p \leq  \lambda_1^{-\alpha_1}\lambda_2^{-\alpha_2} \big\|f\big\|_p
\end{equation}
for all $f\in L^p(\R^d)$.
\end{lemma}
\begin{proof} Before giving the proof of our lemma, we recall the following Young's convolution result (see, e.g, \cite[Theorem 20.18]{Hewitt}): Let $f,\,g:\R^2\to\R$. The convolution of $f$ and $g$ at $(x,y)$ is
$$(f*g)(x,y)\,=\, \int_{\R^2} f((x-x_1,y-y_1)g(x_1,y_1) \,dx_1dy_1$$
provided the integral is defined. \\
Let $p,\,q,\, r\in[1,\infty]$ satisfy
$$\frac{1}{p}+\frac{1}{q}-\frac{1}{r}=1,$$
with the convention $1/\infty=0$. Assume that $f\in L^p(\R^2)$, $g\in L^q(\R^2)$. Then
\begin{enumerate}
  \item The function $(x_1,y_1) \mapsto f(x-x_1,y-y_1)g(x_1,y_1)$ belongs to $L^1(\R^2)$ for almost all $(x,y)$.
  \item The function $(x,y)\mapsto(f*g)(x,y)$ belongs to $L^r(\R^2)$.
  \item There exists a constant $c=c_{p,q} \leq 1$, depending on $p$ and $q$ but not on $f$ or $g$, such that
  \begin{equation}\label{Youn}\|f*g\|_r\leq c\cdot\|f\|_p\cdot \|g\|_q.\end{equation}
\end{enumerate}
We return to the proof of our lemma. Obviously $\mathbb{I}^{\alpha,\,\lambda}_\pm$ is linear, and $\mathbb{I}^{\alpha,\,\lambda}_\pm f(s,t)= (f\ast\phi^\pm_{\alpha})(s,t)$ where
\begin{equation}\label{fo}\phi^+_{\alpha}(s,t)= \frac{1}{\Gamma(\alpha_1)\Gamma(\alpha_1)} s^{\alpha_1-1}t^{\alpha_2-1}e^{-(\lambda_1s +\lambda_2 t)} \mathbf{1}_{\{(0,\,\infty)^2\}}(s,t)\end{equation}
and
$$\phi^-_{\alpha}(s,t)=  \frac{1}{\Gamma(\alpha_1)\Gamma(\alpha_1)} (-s)^{\alpha_1-1}(-t)^{\alpha_2-1}e^{\lambda_1s+\lambda_2t} \mathbf{1}_{\{(-\infty,\,0)^2\}}(s,t)$$
for any $\alpha=(\alpha_1,\alpha_2), \,\lambda=(\lambda_1,\lambda_2)> (0,0)$. But
\begin{eqnarray*}
  \|\phi^\pm_{\alpha}\|_1 &=& \frac{1}{\Gamma(\alpha_1)\Gamma(\alpha_1)} \int_0^{+\infty}\int_0^{+\infty} s^{\alpha_1-1}e^{\lambda_1 s}t^{\alpha_2-1}e^{\lambda_2 t}\,dsdt \\
   &=& \frac{1}{\Gamma(\alpha_1)\Gamma(\alpha_1)} \lambda_1^{-\alpha_1} \Gamma(\alpha_1) \lambda_2^{-\alpha_2} \Gamma(\alpha_2)\\
   &=& \lambda_1^{-\alpha_1}\lambda_2^{-\alpha_2}
\end{eqnarray*}
Then, \eqref{Yo} follows from Young's convolution inequality \eqref{Youn}.
\end{proof}
Next, we discuss the relationship between tempered fractional integrals and Fourier transforms.
Recall that the Fourier transform of $f:\R^2\to \R$ is the function $\mathcal{F}[f](s,t)$ defined by
\begin{eqnarray*}
 \displaystyle \mathcal{F}[f](s,t) &=& \frac{1}{2\pi}\int_{\R^2} e^{i(s\xi_1+t\xi_2)} f(\xi_1,\xi_2)\,d\xi_1d\xi_2
\end{eqnarray*}
\begin{lemma}\label{kes}
For any  $\alpha=(\alpha_1,\alpha_2), \,\lambda=(\lambda_1,\lambda_2)> (0,0)$ we have
$$\mathcal{F}\big[\mathbb{I}^{\alpha,\,\lambda}_\pm f\big] (x,y)=\mathcal{F}[f](x,y)(\lambda_1\pm ix)^{-\alpha_1}(\lambda_2\pm iy)^{-\alpha_2}$$
for all $f\in L^1(\R^2)$ and all $f\in L^2(\R^2)$.
\end{lemma}
\begin{proof}
The function $\phi^+_{\alpha}$ in \eqref{fo} has Fourier transform
\begin{eqnarray*}\mathcal{F}[\phi^+_{\alpha}](x,y)&=& \frac{1}{2\pi\Gamma(\alpha_1)\Gamma(\alpha_2)}\int_{\R^2}e^{i(x\xi_1+y\xi_2)}\xi_1^{\alpha_1-1}\xi_2^{\alpha_2-1}
e^{-(\lambda_1\xi_1+\lambda_2\xi_2)}\\
&&\times\mathbf{1}_{\{(0,\infty)^2\}}(\xi_1,\xi_2)\,d\xi_1d\xi_2\\
&=& \frac{1}{2\pi\Gamma(\alpha_1)\Gamma(\alpha_2)} \int_0^\infty e^{i\xi_1 x}\xi_1^{\alpha_1-1}e^{-\lambda_1\xi_1}d\xi_1 \times \int_0^\infty e^{i\xi_2 y}\xi_2^{\alpha_2-1}e^{-\lambda_2\xi_2}d\xi_2  \\
&=& \frac{1}{2\pi}(\lambda_1+ ix)^{-\alpha_1}(\lambda_2+ iy)^{-\alpha_2}.
\end{eqnarray*}
Now, we give the analog of the two-parameter convolution theorem (in $\R$ one can see, e.g., \cite[Section 15.5]{kind1} and \cite[Chapter 6]{Fouriert}).
Let $f,\,g\in L^1(\R^2)$, it is not easy to show $f\ast g\in L^1(\R^2)$ has Fourier transform  $2\pi \mathcal{F}[f](x,y)\mathcal{F}[g](x,y)$. Then, it follows that
$$\mathcal{F}\big[\mathbb{I}^{\alpha,\,\lambda}_+ f\big] (x,y)=(f\ast\phi^+_{\alpha})(x,y) =\mathcal{F}[f](x,y)(\lambda_1+ ix)^{-\alpha_1}(\lambda_2+ iy)^{-\alpha_2}.$$
Similarly, we prove that
$$\mathcal{F}\big[\mathbb{I}^{\alpha,\,\lambda}_- f\big] (x,y)=(f\ast\phi^-_{\alpha})(x,y) =\mathcal{F}[f](x,y)(\lambda_1- ix)^{-\alpha_1}(\lambda_2- iy)^{-\alpha_2}.$$
If $f \in L^2(\R^2)$, approximated by the $L^1$ function $f(x,y)\mathbf{1}_{[-n_1,n_1]\times [-n_2,n_2]}(x,y)$ and let $n_1,\,n_2\to\infty$.
\end{proof}
\subsection{Two-parameter tempered fractional derivatives}
In this subsection, we consider the inverse operators of the two-parameter tempered fractional integrals, which are called two-parameter tempered fractional derivatives. For our purposes, we only require derivatives of order $\alpha$ such that $0<\alpha_1,\,\alpha_1<1$, and this simplifies the presentation.
\begin{definition}
The positive and negative tempered fractional derivatives of a function $f:\R^2\to\R$ are defined as
{\footnotesize\begin{align}
&\quad\mathbb{D}^{\alpha,\,\lambda}_+f(s,t)\nonumber\\
&= \lambda_1^{\alpha_2}\lambda_2^{\alpha_2}f(s,t)+\frac{\alpha_1\alpha_2}{\Gamma(1-\alpha_1)\Gamma(1-\alpha_2)} \int_{-\infty}^s\int_{-\infty}^t \frac{ f(s,t)-f(u,v)}{(s-u)^{\alpha_1+1}(t-v)^{\alpha_2+1}} e^{-\lambda_1(s-u)}e^{-\lambda_2(t-v)}\,dudv
\end{align}}
and
{\footnotesize
\begin{align}
&\quad\mathbb{D}^{\alpha,\,\lambda}_-f(s,t)\nonumber\\
&= \lambda_1^{\alpha_2}\lambda_2^{\alpha_2}f(s,t)+\frac{\alpha_1\alpha_2}{\Gamma(1-\alpha_1)\Gamma(1-\alpha_2)} \int_{-\infty}^s\int_{-\infty}^t \frac{ f(s,t)-f(u,v)}{(u-s)^{\alpha_1+1}(v-t)^{\alpha_2+1}} e^{-\lambda_1(u-s)}e^{-\lambda_2(v-t)}\,dudv,
\end{align}}
respectively, for any $0 < \alpha_1,\,\alpha_2 < 1$ and any $\lambda_1,\,\lambda_2> 0$.
\end{definition}
 Two-parameter tempered fractional derivatives cannot be defined pointwise for all functions $f\in L^p(\R^2)$, since we need $|f(s,t)-f(u,v)|\to0$ fast enough to counter the singularity of the denominator $(s-u)^{\alpha_1+1}(t-v)^{\alpha_2+1}$ as $u\to s$ and $v\to t$. We can extend the definitions of the two-parameter tempered fractional derivatives to a suitable class of functions in $L^2(\R^2)$. For any $\alpha=(\alpha_1,\alpha_2) > (0,0)$ and
$\lambda=(\lambda_1,\lambda_2) > (0,0)$ we may define the fractional Sobolev space
{\small\begin{equation*}
W^{\alpha,\,2}(\R^2)\,:=\, \Bigg\{f\in L^2(\R^2)\,:\,\, \int_{\R^2}(\lambda_1^2+\omega_1^2)^{\alpha_1}(\lambda_2^2+\omega_2^2)^{\alpha_2}\Big\vert\mathcal{F}[f](\omega_1,\omega_2)\Big\vert^2\,d\omega_1d\omega_2<\infty\Bigg\},
\end{equation*}}
which is a Banach space with norm $$\|f\|_{\alpha,\,\lambda} = \|(\lambda_1^2+\omega_1^2)^{\alpha_1}(\lambda_2^2+\omega_2^2)^{\alpha_2}\Big\vert\mathcal{F}[f](\omega_1,\omega_2)\Big\vert^2\|_2.$$
The space $W^{\alpha,\,2}(\R^2)$
is the same for any $\lambda_1,\,\lambda_2> 0$ (typically we take $\lambda_1=\lambda_2=1$) and all the norms $\|f\|_{\alpha,\,\lambda}$ are
equivalent, since $(1+\omega_1^2)(1+\omega_2^2)\leq (\lambda_1^2+\omega_1^2)(\lambda_2^2+\omega_2^2)\leq \lambda_1^2\lambda_2^2(1+\omega_1^2)(1+\omega_2^2)$ for all $\lambda_1,\,\lambda_2\geq 1$, and $(\lambda_1^2+\omega_1^2)(\lambda_2^2+\omega_2^2)\leq (1+\omega_1^2)(1+\omega_2^2)\leq \lambda_1^{-2}\lambda_2^{-2}(1+\omega_1^2)(1+\omega_2^2)$ for all $0<\lambda_1,\,\lambda_2<1$.
\begin{definition}\label{dorsaf2}
The positive (resp., negative) two-parameter tempered fractional derivative $\mathbb{D}^{\alpha,\,\lambda}_\pm f(s,t)$ of a
function $f\in W^{\mathbf{\alpha},\,2}(\R^2)$ is defined as the unique element of $L^2(\R^2)$ with Fourier transform $\mathcal{F}[f](x,y)(\lambda_1 \pm ix)^{\alpha_1}(\lambda_2 \pm iy)^{\alpha_2}$ for any $\alpha_1,\,\alpha_2 > 0$ and any $\lambda_1,\,\lambda_1 >0.$
\end{definition}
\begin{lemma}\label{dorsaf}
For any $\alpha_1,\,\alpha_2 > 0$ and any $\lambda_1,\,\lambda_1 >0$, we have
\begin{equation}\label{rad}\mathbb{D}^{\alpha,\,\lambda}_\pm \mathbb{I}^{\alpha,\,\lambda}_\pm f(s,t)= f(s,t)\end{equation}
for any function $f\in L^2(\R^2)$, and
\begin{equation}\label{rad1}\mathbb{I}^{\alpha,\,\lambda}_\pm \mathbb{D}^{\alpha,\,\lambda}_\pm f(s,t)= f(s,t)\end{equation}
for any $f \in W^{\alpha,\,2}(\R^2)$.
\end{lemma}
\begin{proof}
Given $f\in L^2(\R^2)$, note that $g(s,t)=\mathbb{I}^{\alpha,\,\lambda}_\pm f(s,t)$ satisfies, by Lemma \ref{kes},
$$\mathcal{F}[g](x,y)=\mathcal{F}[f](x,y)(\lambda_1\pm  ix)^{-\alpha_1} (\lambda_2\pm iy)^{-\alpha_2}.$$
Then, it follows easily that $g\in W^{\alpha,\,2}(\R^2)$.
Definition \ref{dorsaf2} implies that
\begin{eqnarray*}
\mathcal{F}[\mathbb{D}^{\alpha,\,\lambda}_\pm \mathbb{I}^{\alpha,\,\lambda}_\pm f](x,y)&=& \mathcal{F}[\mathbb{D}^{\alpha,\,\lambda}_\pm g](x,y)\\
&=& \mathcal{F}[g](x,y)(\lambda_1\pm  ix)^{-\alpha_1} (\lambda_2\pm iy)^{-\alpha_2}\\
&=& \mathcal{F}[f](x,y).
\end{eqnarray*}
Then, \eqref{rad} follows using the uniqueness of the Fourier transform. Similarly, we can prove \eqref{rad1}.
\end{proof}
\section{Wiener integrals with respect to the two-parameter tempered Hermite field of order one}
\label{sect6}
Recall that the two-parameter tempered Hermite field of order one is given by:
\begin{align}
& \qquad Z^{1,\,H_1,H_2}_{\lambda_1,\lambda_2}(s,t)\nonumber\\
&=\int_{\mathbb{R}^2}^\prime\int_0^t \int_0^s (a-x)_+^{H_1-\frac{3}{2}}e^{-\lambda_1(a-x)_+}(b-y)_+^{H_2-\frac{3}{2}}e^{-\lambda_2(b-y)_+}\,da\,db\,dW(x, y),\label{tHs1}
\end{align}
where $H_1,\,H_2>\frac{1}{2}$ and $\lambda_1,\lambda_2>0$.

In this section, we will develop the theory of Wiener integrals with respect to the two-parameter tempered Hermite field of order one. We consider two cases:
\begin{itemize}
  \item $\frac{1}{2}<H_1,\,H_2<1$, $\lambda_1,\,\lambda_2>0$
  \item $H_1,\,H_2>1$, $\lambda_1,\,\lambda_2>0$.
\end{itemize}
\subsection{Case 1: $\frac{1}{2}<H_1,\,H_2<1$ and $\lambda_1,\,\lambda_2>0$}
 We first establish a link between $Z^{1,\,H_1,H_2}_{\lambda_1,\lambda_2}$ and the two-parameter tempered fractional calculus developed in the previous section.
\begin{lemma}
For a two-parameter tempered Hermite field of order one given by \eqref{tHs1} with $\lambda_1,\,\lambda_2>0$, we have:
\begin{equation}\label{Oneh}
 Z^{1,\,H_1,H_2}_{\lambda_1,\lambda_2}(s,t)\,=\,  \Gamma(H_1-\frac{1}{2})\Gamma(H_2-\frac{1}{2})\int_{\mathbb{R}^2}^\prime \Big(\mathbb{I}^{\beta,\,\lambda}_- \mathbf{1}_{[0,s]\times[0,t]}\Big)(x,y) dW(x,y),
 \end{equation}
 where $\beta=(H_1-\frac{1}{2},H_2-\frac{1}{2})$ such that $H_1,\,H_2>\frac{1}{2}$.
\end{lemma}
\begin{proof}
Write the kernel function in \eqref{tHs1} in the form
\begin{eqnarray*}
   h_{s,\,t}(x,y)&=& \int_0^s \int_0^t (a-x)_+^{H_1-\frac{3}{2}}e^{-\lambda_1(a-x)_+}(b-y)_+^{H_2-\frac{3}{2}}e^{-\lambda_2(b-y)_+}\,da\,db \\
   &=& \int_{\R^2} (a-x)_+^{H_1-\frac{3}{2}}e^{-\lambda_1(a-x)_+}(b-y)_+^{H_2-\frac{3}{2}}e^{-\lambda_2(b-y)_+}\mathbf{1}_{[0,s]\times[0,t]}(a,b)\,da\,db \\
   &=& \Gamma(H_1-\frac{1}{2}) \Gamma(H_2-\frac{1}{2})\Big(\mathbb{I}^{\beta,\,\lambda}_- \mathbf{1}_{[0,s]\times[0,t]}\Big)(x,y),
\end{eqnarray*}
where $\beta=(H_1-\frac{1}{2},H_2-\frac{1}{2})$.
\end{proof}
Next, we discuss a general construction for stochastic integrals with respect to $Z^{1,\,H_1,H_2}_{\lambda_1,\lambda_2}$. Recall how we classically defined Wiener integrals with respect to the Brownian field: first we define it for elementary functions and establish the isometry property, then we extend the integral for general functions via isometry.

Denote $\mathcal{E}$ the family of elementary functions on $\R^2$ of the form
{\small
\begin{equation}\label{yassin}
  f(x,y) = \sum_{\ell=1}^n a_\ell\mathbf{1}_{(s_\ell,\,s_{\ell+1}]\times (t_\ell,\,t_{\ell+1}]}(x,y) ,\,\,s_\ell<s_{\ell+1},\,t_\ell<t_{\ell+1}\, a_\ell\in\R,\,\,\ell=1,\ldots, n.
\end{equation}}
For functions like $f$ above we can naturally define its Wiener integral with respect to the two-parameter tempered Hermite field of order one as:
{\small\begin{align}
&\quad \mathcal{I}^{\alpha,\,\lambda}(f)\nonumber\\
&=\int_{\R^2}f(x,y)\,dZ^{1,\,H_1,H_2}_{\lambda_1,\lambda_2}(x,y)\nonumber\\
&=\sum_{\ell=1}^n a_\ell \Big[Z^{k,\,H_1,H_2}_{\lambda_1,\lambda_2}(s_{\ell+1}, t_{\ell+1})-Z^{k,\,H_1,H_2}_{\lambda_1,\lambda_2}(s_{\ell+1},t_\ell)-Z^{k,\,H_1,H_2}_{\lambda_1,\lambda_2}(s_\ell,t_{\ell+1})+
Z^{k,\,H_1,H_2}_{\lambda_1,\lambda_2}(s_\ell,t_\ell)\Big],\nonumber
\end{align}}
where $\beta=(H_1-\frac{1}{2},H_2-\frac{1}{2})$. Then it follows immediately from \eqref{Oneh} that for $f\in\mathcal{E}$, the space of elementary functions, the stochastic integral
\begin{eqnarray}
\mathcal{I}^{\alpha,\,\lambda}(f) &=&\int_{\R^2}f(x,y)\,dZ^{1,\,H_1,H_2}_{\lambda_1,\lambda_2}(x,y)\nonumber\\
 &=& \Gamma(H_1-\frac{1}{2})\Gamma(H_2-\frac{1}{2})\int_{\mathbb{R}^2}^\prime \Big(\mathbb{I}^{\beta,\,\lambda}_- f\Big)(a,b) dW(a,b)\nonumber
\end{eqnarray}
is a Gaussian random field with mean zero, such that for any $f,g\in\mathcal{E}$ we have
\begin{align}
    &\langle  \mathcal{I}^{\alpha,\,\lambda}(f),\,  \mathcal{I}^{\alpha,\,\lambda}(g)\rangle_{L^2(\Omega)}\nonumber\\
    &\quad =\mathbb{E}\Big[\int_{\R^2}f(x,y)\,dZ^{1,\,H_1,H_2}_{\lambda_1,\lambda_2}(x,y)\int_{\R^2}g(x,y)\,dZ^{1,\,H_1,H_2}_{\lambda_1,\lambda_2}(x,y)\Big]\nonumber\\
    &\quad =\Gamma(H_1-\frac{1}{2})^2\Gamma(H_2-\frac{1}{2})^2\int_{\mathbb{R}^2} \Big(\mathbb{I}^{\beta,\,\lambda}_- f\Big)(a,b)\Big(\mathbb{I}^{\beta,\,\lambda}_- g\Big)(a,b)\,dadb .\label{Tu}
\end{align}
The linear space of Gaussian random variables $\big\{ \mathcal{I}^{\alpha,\,\lambda}(f),\,\,f\in\mathcal{E}\big\}$ is contained in the larger linear space
\begin{equation*}
  \overline{\text{Sp}}(Z^{1,\,H_1,H_2}_{\lambda_1,\lambda_2})\,=\,\big\{X: \mathcal{I}^{\alpha,\,\lambda}(f_n)\to X\,\, \text{in $L^2(\Omega)$ for a sequence $f_n$ in $\mathcal{E}$}\big\}.
\end{equation*}
An element $X\in \overline{\text{Sp}}(Z^{1,\,H_1,H_2}_{\lambda_1,\lambda_2})$ is mean zero Gaussian with variance
$$\text{Var}(X)\,=\,\lim_{n\to\infty} \text{Var}[\mathcal{I}^{\alpha,\,\lambda}(f_n)],$$
and $X$ can be associated with an equivalence class of sequences of elementary functions $(f_n)$ such that $\mathcal{I}^{\alpha,\,\lambda}(f_n)\to X\,\, \text{in $L^2(\R^2)$}$. If $[f_X]$ denotes this class, then $X$ can be
written in an integral form as
 \begin{equation}\label{be}
   X\,=\,\int_{\R^2} [f_X]\, dZ^{1,\,H_1,H_2}_{\lambda_1,\lambda_2}
 \end{equation}
and the right-hand side of \eqref{be} is called the stochastic integral with respect to the two-parameter tempered Hermite field of order one $Z^{1,\,H_1,H_2}_{\lambda_1,\lambda_2}$ on $\R^2$.\\
Recall that for the case of Brownian field: $\lambda_1=\lambda_2=0$ and $H_1=H_2\frac{1}{2}$, $\mathcal{I}^{\alpha,\,\lambda}(f_n)\to X$ along with the following It\^{o} isometry
\begin{equation*}
  \langle \mathcal{I}(f),\,\mathcal{I}(g)\rangle_{L^2(\Omega)}\,=\,\text{Cov}[\mathcal{I}(f),\,\mathcal{I}(g)]\,=\,\int_{\R^2}f(x,y)g(x,y)\,dxdy\,=\,\langle f,\,g\rangle_{L^2(\R^2)}
\end{equation*}
implies that $(f_n)$ is a Cauchy sequence, and then since $L^2(\R^2)$ is a (complete) Hilbert space, there exists a unique $f\in L^2(\R^2)$ such that $f_n\to f$ in $L^2(\R^2)$, and we can write
\begin{equation*}
  X\,=\, \int_{\R^2}f(x,y)\,d W(x,y).
\end{equation*}
However, if the space of integrands is not complete, then the situation is more complicated. Here we investigate stochastic integral with
respect to the two-parameter tempered Hermite field of order one based on time domain representation. Equation \eqref{Tu} suggests the
appropriate space of integrands for the two-parameter tempered Hermite sheet of order one, in order to obtain a nice isometry that maps into the space $\overline{\text{Sp}}(Z^{1,\,H_1,H_2}_{\lambda_1,\lambda_2})$ of stochastic integrals.
\begin{theorem}\label{dorra3}
Given $\frac{1}{2}<H_1,\,H_2<1$ and $\lambda_1,\,\lambda_2>0$, the class of functions
\begin{equation*}
  \mathcal{H}_1\,:=\, \Bigg\{f\in L^2(\R^2):\,\, \int_{\R^2}\Big\vert \Big(\mathbb{I}^{\beta,\,\lambda}_- f\Big)(a,b)\Big\vert^2\,da\,db<\infty \Bigg\},
\end{equation*}
is a linear space with the inner product
\begin{equation}\label{Atef2}
  \langle f,\,g\rangle_{\mathcal{H}_1}\,:=\, \langle F, \, G\rangle_{L^2(\R^2)}
\end{equation}
where
\begin{equation*}
  F(a,b)\,=\,\Gamma(H_1-\frac{1}{2})\Gamma(H_2-\frac{1}{2})\Big(\mathbb{I}^{\beta,\,\lambda}_- f\Big)(a,b)
\end{equation*}
and
\begin{equation}\label{}
   G(a,b)\,=\, \Gamma(H_1-\frac{1}{2})\Gamma(H_2-\frac{1}{2})\Big(\mathbb{I}^{\beta,\,\lambda}_- g\Big)(a,b),
\end{equation}
where $\beta=(H_1-\frac{1}{2},H_2-\frac{1}{2})$ and $\lambda=(\lambda_1,\lambda_2)$. The set of elementary functions $\mathcal{E}$ is dense in the space $\mathcal{H}_1$. The space $\mathcal{H}_1$ is not complete.
\end{theorem}
\begin{proof}
To show that $\mathcal{H}_1$ is an inner product space, we will check that $ \langle f,\,f\rangle_{\mathcal{H}_1}=0$ implies $f=0$ almost everywhere. \\
If $ \langle f,\,f\rangle_{\mathcal{H}_1}=0$ then we have $ \langle F,\,F\rangle_{L^2(\R^2)}=0$, which implies that
\begin{equation*}\label{}
  F(a,b)\,=\,\Gamma(H_1-\frac{1}{2})\Gamma(H_2-\frac{1}{2})\Big(\mathbb{I}^{\beta,\,\lambda}_- f\Big)(a,b)=0,\,\,\text{ for almost every $(a,b)\in\R^2$}.
\end{equation*}
Then,
\begin{equation}\label{vvv}
  \Big(\mathbb{I}^{\beta,\,\lambda}_- f\Big)(a,b)=0,\,\,\text{ for almost every $(a,b)\in\R^2$}.
\end{equation}
Apply $\mathbb{D}^{\beta,\, \lambda}_-$, $\beta=(H_1-\frac{1}{2},H_2-\frac{1}{2})$, to both sides of equation \eqref{vvv} and use Lemma \ref{dorsaf} to get $f(a,b)=0$ for
almost every $(a,b)\in\R^2$, and hence $\mathcal{H}_1$ is an inner product space.

Next, we want to show that the set of elementary functions $\mathcal{E}$ is dense in $\mathcal{H}_1$. For
any $f\in \mathcal{H}_1$, we also have $f \in L^2(\R^2)$, and hence there exists a sequence of elementary
functions $(f_n)$ in $L^2(\R^2)$ such that $\|f-f_n\|_{L^2(\R^2)}$. But
$$\|f-f_n\|_{\mathcal{H}_1}= \langle f-f_n,\,f-f_n\rangle_{\mathcal{H}_1} =\langle F-F_n,\, F-F_n\rangle_{L^2(\R^2)}=\|F-F_n\|_{L^2(\R^2)},$$
where
\begin{equation}\label{}
  F_n(a,b)\,=\,\Gamma(H_1-\frac{1}{2})\Gamma(H_2-\frac{1}{2})\Big(\mathbb{I}^{\beta,\,\lambda}_- f_n\Big)(a,b).
\end{equation}
Lemma \ref{vvv1} implies that
$$\|f-f_n\|_{\mathcal{H}_1}= \|F-F_n\|_{L^2(\R^2)} =\|\mathbb{I}^{\beta,\,\lambda}_-(f-f_n)\|_{L^2(\R^2)}\leq C \|f-f_n\|_{L^2(\R^2)}$$
for some $C>0$, and since $\|f-f_n\|_{L^2(\R^2)}\to 0$, it follows that the set of elementary functions is dense in $\mathcal{H}_1$.

Finally, we provide an example to show that $\mathcal{H}_1$ is not complete. Proceeding as \cite[Proof of Theorem 3.1]{PipTaq2000:Int} the functions
$$\widehat{f}_n(x,y) =|xy|^{-p}\mathbf{1}_{\{1<|x|,\,|y|<n\}}(x,y),\,\,p>0,$$
are in $L^2(\R^2)$, $\overline{\widehat{f}_n(x,y)} = \widehat{f}_n(-x,-y)$, and hence they are the Fourier transforms of the function $f_n\in L^2(\R^2)$. Apply Lemma \ref{kes} to see that $F_n(x,y)=\Gamma(H_1-\frac{1}{2})\Gamma(H_2-\frac{1}{2})\big(\mathbb{I}^{\beta,\,\lambda}_-f\big)(x,y)$ have Fourier transform
\begin{equation}
\mathcal{F}[F_n](x,y)\,=\, \Gamma(H_1-\frac{1}{2})\Gamma(H_2-\frac{1}{2})(\lambda_1+ix)^{\frac{1}{2}-H_1}(\lambda_2+iy)^{\frac{1}{2}-H_2}\widehat{f}_n(x,y).
\end{equation}
Since $\frac{1}{2}-H_1,\,\frac{1}{2}-H_1<0$, it follows that
\begin{align*}
    \|F_n\|^2_2&=\|\mathcal{F}[F_n]\|^2_2= \Gamma(H_1-\frac{1}{2})^2\Gamma(H_2-\frac{1}{2})^2\\
    &\quad\times \int\limits_{-\infty}^\infty\int_{-\infty}^\infty\vert \widehat{f}_n(x,y)\vert^2(\lambda_1^2+x^2)^{\frac{1}{2}-H_1}(\lambda_2^2+y^2)^{\frac{1}{2}-H_2}\,dxdy<\infty
\end{align*}
 which shows that $f_n\in \mathcal{H}_1$. Now it is easy to check that $f_n-f_m\to0$ in $\mathcal{H}_1$, as $n,\,m\to\infty$, whenever $p>\max (1-H_1,\,1-H_2)$, so that $(f_n)$ is a Cauchy sequence. Choose $p=\frac{1}{2}$ and suppose that there exists some $f\in\mathcal{H}_1$ such that $\|f-f_n\|_{\mathcal{H}_1}\to 0$ as $n\to\infty$.
Then
\begin{equation*}
  \int_{-\infty}^\infty\int_{-\infty}^\infty\vert \widehat{f}_n(x,y)-\widehat{f}(x,y)\vert^2(\lambda_1^2+x^2)^{\frac{1}{2}-H_1}(\lambda_2^2+y^2)^{\frac{1}{2}-H_2}\,dxdy\to 0,
\end{equation*}
as $n\to\infty$, and since, for any given $m\geq1$, the value of $\widehat{f}_n(x,y)$ does not vary with $n > \max(m_1,\,m_2)$ whenever $(x,y)\in[-m_1,m_1]\times[-m_2,m_2]$, it follows that $\widehat{f}(x,y)=|xy|^{-\frac{1}{2}}\mathbf{1}_{\{|x|,\,|y|>1\}}$ on any such
interval. Since $m_1,\,m_2$ are arbitrary, it follows that $\widehat{f}(x,y)=|xy|^{-\frac{1}{2}}\mathbf{1}_{\{|x|,\,|y|>1\}}$, but this function
is not in $L^2(\R^2)$, so $\widehat{f}(x,y)\notin \mathcal{H}_1$, which is a contradiction. Hence $\mathcal{H}_1$ is not complete, and this completes the proof.
\end{proof}
We now define the stochastic integral with respect to the two-parameter tempered Hermite field for any function in $\mathcal{H}_1$ in the case where $\frac{1}{2} < H_1,\,H_2 < 1$.
\begin{definition}\label{dorra}
For any $\frac{1}{2} < H_1,\,H_2 < 1$ and $\lambda_1,\,\lambda_2 > 0$, we define
\begin{equation}\label{vvv2}
  \int_{\R^2} f(x,y)\, dZ^{1,\,H_1,H_2}_{\lambda_1,\lambda_2}(x,y)\,:=\, \Gamma(H_1-\frac{1}{2})\Gamma(H_2-\frac{1}{2})\int_{\R^2} \Big( \mathbb{I}^{\beta,\,\lambda}_- f\Big)(x,y)\,dW(x,y),
\end{equation}
where $\beta=(H_1-\frac{1}{2},H_2-\frac{1}{2})$, for any $f\in\mathcal{H}_1$.
\end{definition}
\begin{theorem}\label{aaaaaaaa}
For any  $\frac{1}{2} < H_1,\,H_2 < 1$ and $\lambda_1,\,\lambda_2 > 0$, the stochastic integral  in \eqref{vvv2} is
an isometry from $\mathcal{H}_1$ into $\overline{\text{Sp}}(Z^{1,\,H_1,H_2}_{\lambda_1,\lambda_2})$. Since $\mathcal{H}_1$ is not complete, these two spaces are not isometric.
\end{theorem}
\begin{proof}
It follows from Lemma \ref{vvv1} that the stochastic integral \eqref{vvv2} is well-defined for any $f\in \mathcal{H}_1$.
The extension of \cite[Proposition 2.1]{PipTaq2000:Int} to $d=2$ is natural and it reads as follows: if $\mathcal{D}$ is an inner
product space such that $(f,g)_{\mathcal{D}} = \langle \mathcal{I}^{\alpha,\,\lambda}(f), \mathcal{I}^{\alpha,\,\lambda}(g)\rangle_{L^2(\Omega)}$ for all $f,\,g \in\mathcal{E}$ ($\mathcal{E}$ the family of elementary functions on $\R^2$ of the form \eqref{yassin}),  and if $\mathcal{E}$ is
dense in $\mathcal{D}$, then there is an isometry between $\mathcal{D}$ and a linear subspace of $\overline{\text{Sp}}(Z^{1,\,H_1,H_2}_{\lambda_1,\lambda_2})$ that
extends the map $f\to \mathcal{I}^{\alpha,\,\lambda}(f)$ for $f\in\mathcal{E}$, and furthermore, $\mathcal{D}$ is isometric to $\overline{\text{Sp}}(Z^{1,\,H_1,H_2}_{\lambda_1,\lambda_2})$ itself if and only if $\mathcal{D}$ is complete.
Using the It\^{o} isometry and the definition \ref{dorra}, it follows from \eqref{Atef2} that for any $f,\, g \in\mathcal{H}_1$ we have
$$\langle f,\,g\rangle_{\mathcal{H}_1}\,=\, \langle F,\,G\rangle_{L^2(\R^2)}\,=\, \langle \mathcal{I}^{\alpha,\,\lambda}(f),\, \mathcal{I}^{\alpha,\,\lambda}(g)\rangle_{L^2(\Omega)},$$
and then the result follows from Theorem \ref{dorra3}.
\end{proof}
We now apply the spectral domain representation of two-parameter tempered Hermite field given in Section \ref{sect4} to investigate the stochastic integral with respect to $Z^{1,\,H_1,H_2}_{\lambda_1,\lambda_2}$. First, recall that the Fourier transform of an indicator function is
\begin{eqnarray*}
  \mathcal{F}[\mathbf{1}_{[0,s]\times[0,t]}](\xi,\omega)&=& \frac{1}{2\pi}\int_0^s\int_0^s e^{i\xi x}e^{i\omega y}\,dx\,dy \\
   &=& -\frac{1}{2\pi}\frac{(e^{it\xi}-1)(e^{is\omega}-1)}{ \xi\omega}.
\end{eqnarray*}
Apply this to write this spectral domain representation of the two-parameter tempered Hermite field in the form
\begin{eqnarray}
  Z^{1,\,H_1,H_2}_{\lambda_1,\lambda_2}(s,t) &=& \Gamma(H_1-\frac{1}{2})\Gamma(H_2-\frac{1}{2})\int_{\mathbb{R}^2}^{\prime\prime} \mathcal{F}[\mathbf{1}_{[0,s]\times[0,t]}](\xi,\omega) \nonumber\\
   && \times (\lambda_1+i\xi)^{-(\frac{1}{2}-\frac{1-H_1}{k})}(\lambda_2+i\omega)^{-(\frac{1}{2}-\frac{1-H_2}{k})}\,
   \widehat{W}(d\xi d\omega). \nonumber
\end{eqnarray}
It follows easily that for any elementary function \eqref{yassin} we may write
\begin{eqnarray}
 \mathcal{I}^{\alpha,\,\lambda}(f) &=& \Gamma(H_1-\frac{1}{2})\Gamma(H_2-\frac{1}{2})\int_{\mathbb{R}^2}^{\prime\prime} \mathcal{F}[f](\xi,\omega) \nonumber\\
   && \times (\lambda_1+i\xi)^{\frac{1}{2}-H_1}(\lambda_2+i\omega)^{\frac{1}{2}-H_2}\,
   \widehat{W}(d\xi d\omega), \nonumber
\end{eqnarray}
and then for any elementary functions $f$ and $g$ we have
{\small\begin{align*}
&\langle \mathcal{I}^{\alpha,\,\lambda}(f),\,\mathcal{I}^{\alpha,\,\lambda}(g)\rangle_{L^2(\Omega)}=\Gamma(H_1-\frac{1}{2})\Gamma(H_2-\frac{1}{2})\\
&\quad\times \int\limits_{\mathbb{R}^2}^{\prime\prime} \mathcal{F}[f](\xi,\omega) \overline{\mathcal{F}[g](\xi,\omega)} (\lambda_1^2+\xi^2)^{\frac{1}{2}-H_1}(\lambda_2^2+\omega^2)^{\frac{1}{2}-H_2}\,
   \widehat{W}(d\xi d\omega).
\end{align*}}
\begin{theorem}\label{aissa1}
For any $\frac{1}{2} < H_1,\,H_2 < 1$ and $\lambda_1,\,\lambda_2 > 0$, the class of functions
\begin{equation}\label{}
  \mathcal{H}_2\,=\,\Bigg\{f\in L^2(\R^2):\,\, \int \Big\vert \mathcal{F}[f](\xi,\omega)\Big\vert^2 \, (\lambda_1^2+\xi^2)^{\frac{1}{2}-H_1}(\lambda_2^2+\omega^2)^{\frac{1}{2}-H_2} \,d\xi\,d\omega<\infty\Bigg\},
\end{equation}
is a linear space with the inner product
\begin{eqnarray*}
    \langle f,\, g\rangle_{\mathcal{H}_2}&=&\Gamma(H_1-\frac{1}{2})^2\Gamma(H_2-\frac{1}{2})^2\nonumber\\
    &&\times\int_{\mathbb{R}^2}^{\prime\prime} \mathcal{F}[f](\xi,\omega) \overline{\mathcal{F}[g](\xi,\omega)} (\lambda_1^2+\xi^2)^{\frac{1}{2}-H_1}(\lambda_2^2+\omega^2)^{\frac{1}{2}-H_2}\,
   \widehat{W}(d\xi ,d\omega).
\end{eqnarray*}

The set of elementary functions $\mathcal{E}$ is dense in the space $\mathcal{H}_2$. The space $\mathcal{H}_2$ is not
complete.
\end{theorem}
\begin{proof}
Since $H_1,\,H_2> \frac{1}{2}$, the function $(\lambda_1^2+\xi^2)^{\frac{1}{2}-H_1}(\lambda_2^2+\omega^2)^{\frac{1}{2}-H_2}$ is bounded by a constant $C(H_1,H_2, \lambda_1,\lambda_2)$ that depends only on $H_1,\,H_2,\,\lambda_1$ and $\lambda_2$, so for any $f\in L^2(\R^2)$ we have
{\footnotesize
\begin{equation}\label{}
  \int_{\R^2} \vert \mathcal{F}[f](\xi,\omega)\vert^2(\lambda_1^2+\xi^2)^{\frac{1}{2}-H_1}(\lambda_2^2+\omega^2)^{\frac{1}{2}-H_2}\,d\xi d\omega\leq C(H_1,H_2, \lambda_1,\lambda_2) \int_{\R^2} \vert \mathcal{F}[f](\xi,\omega)\vert^2 \,d\xi d\omega<\infty
\end{equation}}
and hence $f\in\mathcal{H}_2$.
Since $\mathcal{H}_2\subseteq L^2(\R^2)$ by definition, this proves that $L^2(\R^2)$ and $\mathcal{H}_2$
are the same set of functions, and then it follows from Lemma \ref{vvv1} that $\mathcal{H}_1$ and $\mathcal{H}_2$
are the same set of functions. Observe that $\varphi_f= \big(\mathbb{I}^{\beta ,\,\lambda}_-f\big)$, where$\beta=(H_1-\frac{1}{2},H_2-\frac{1}{2})$), is again a function with Fourier transform
$$\mathcal{F}[\varphi_f](\xi,\omega)\,=\, (\lambda_1+i\xi)^{\frac{1}{2}-H_1}(\lambda_2+i\omega)^{\frac{1}{2}-H_2}\mathcal{F}[f](\xi,\omega).$$
Then, it follows from the Plancherel Theorem that
{\small\begin{align*}
    & \langle f,\,g\rangle_{\mathcal{H}_1}\\
    &\,\,=\Gamma(H_1-\frac{1}{2})^2\Gamma(H_2-\frac{1}{2})\langle \varphi_f,\,\varphi_g\rangle_2 \\
    &\,\, =\Gamma(H_1-\frac{1}{2})^2\Gamma(H_2-\frac{1}{2})^2\langle \mathcal{F}[\varphi_f],\,\mathcal{F}[\varphi_g]\rangle_2\\
    &\,\,=\Gamma(H_1-\frac{1}{2})^2\Gamma(H_2-\frac{1}{2})^2\int_{-\infty}^\infty\int_{-\infty}^\infty \mathcal{F}[f](\xi,\omega)\overline{\mathcal{F}[g](\xi,\omega)}(\lambda_1^2+\xi^2)^{\frac{1}{2}-H_1}(\lambda_2^2+\omega^2)^{\frac{1}{2}-H_2}\,d\xi d\omega\\
    &\,\,=\langle f,\, g\rangle_{\mathcal{H}_2},
\end{align*}}
and hence the two inner products are identical.
Then, the conclusions of Theorem \ref{aissa1}
follow from Theorem \ref{dorra3}.
\end{proof}
\begin{definition}
For any $\frac{1}{2} < H_1,\,H_2 < 1$ and $\lambda_1,\,\lambda_2 > 0$, we define
\begin{eqnarray}
 \mathcal{I}^{\alpha,\,\lambda}(f) &=& \Gamma(H_1-\frac{1}{2})\Gamma(H_2-\frac{1}{2})\int_{\mathbb{R}^2}^{\prime\prime} \mathcal{F}[f](\xi,\omega) \nonumber\\
   && \times (\lambda_1+i\xi)^{\frac{1}{2}-H_1}(\lambda_2+i\omega)^{\frac{1}{2}-H_2}\,
   \widehat{W}(d\xi d\omega), \label{vvv3}
\end{eqnarray}
for any $f\in\mathcal{H}_2$.
\end{definition}
\begin{theorem}
For any  $\frac{1}{2} < H_1,\,H_2 < 1$ and $\lambda_1,\,\lambda_2 > 0$, the stochastic integral  in \eqref{vvv3} is
an isometry from $\mathcal{H}_2$ into $\overline{\text{Sp}}(Z^{1,\,H_1,H_2}_{\lambda_1,\lambda_2})$. Since $\mathcal{H}_2$ is not complete, these two spaces are not isometric.
\end{theorem}
\begin{proof}
The proof of Theorem \ref{aissa1} shows that $\mathcal{H}_1$ and $\mathcal{H}_2$ are identical when $H_1,\,H_2 > \frac{1}{2}$.
Then, the result follows immediately from Theorem \ref{aaaaaaaa}.
\end{proof}
\subsection{Case 2: $H_1,\,H_2>1$ and $\lambda_1,\,\lambda_2>0$}
Now, we consider the second case that we mentioned at the beginning of this section. We will show that $Z^{1,\,H_1,H_2}_{\lambda_1,\lambda_2}$ is a continuous semimartingale with a finite variation and hence one can define stochastic integrals $\displaystyle I(f) :=\int f(x,y) Z^{1,\,H_1,H_2}_{\lambda_1,\lambda_2}(dx,dy)$ in the standard manner, via the It\^{o} stochastic calculus.
\begin{theorem}
A two-parameter tempered Hermite field of order one {\small$\{Z^{1,\,H_1,H_2}_{\lambda_1,\lambda_2}(s,t)\}_{s,\,t\geq 0}$} with $H_1,\, H_2 > 1$
and $\lambda_1,\,\lambda_2 > 0$ is a continuous semimartingale with the canonical decomposition
\begin{equation}\label{38}
  Z^{1,\,H_1,H_2}_{\lambda_1,\lambda_2}(s,t)\,=\, \int_0^s\int_0^t M_{H_1,H_2,\lambda_1,\lambda_2}(x,y)\,dxdy
\end{equation}
where
\begin{align*}M_{H_1,H_2,\lambda_1,\lambda_2}(x,y)\,&:=\, \int_{-\infty}^\infty\int_{-\infty}^\infty (x-\xi)_+^{H_1-\frac{3}{2}}(y-\omega)_+^{H_2-\frac{3}{2}}\\ &\quad\times e^{-\lambda_1(x-\xi)_+}e^{-\lambda_2(y-\omega)_+}W(d\xi,d\omega).\end{align*}
Moreover, $\{Z^{1,\,H_1,H_2}_{\lambda_1,\lambda_2}(s,t)\}_{s\,t\geq 0}$ is a finite variation process.
\end{theorem}
\begin{proof} Then proof is similar to that of \cite[Theorem 3.2]{MeeSab2014:SI}. Let $(W(s,t),\, s,t\in\R)$ be a two-parameter Brownian field and let $\{\mathcal{F}^W_{s,t}\}_{s,t\leq 0}$ be the $\sigma$-algebra generated by $\{W(x,y): \,0\leq x\leq s,\, 0\leq y\leq t\}$.\\
Given a function $g:\R^2\to\R$ such that $g(s,t)=0$ for all $s ,\,t< 0$, and
\begin{equation}\label{hg}
  g(s,t)\,=\,  C+\int_0^s\int_0^sh(x,y)\,dxdy,\,\,\,\text{for all$s,\,t>0$},
\end{equation}
for $C\in\R$ and some function $h\in L^2(\R^2)$.\\
A natural extension of  \cite[Theorem 3.9]{Che2004:Gaussian} to $\R^2$ shows that the Gaussian stationary increment process
\begin{equation}\label{mar}
  Y_{s,t}^g:=\int_{\R^2}\Big[ g(t-u,s-v)-g(-u,s-v)-g(t-u,-v)+g(-u,-v)\Big]\,dudv
\end{equation}
is a continuous $\{\mathcal{F}^W_{s,t}\}_{s,t\geq 0}$ semimartingale with canonical decomposition
\begin{equation}\label{}
  Y_{s,t}^g\,=\, g(0,0)W_{s,t}+\int_0^s\int_{-\infty}^x\int_0^t\int_{-\infty}^y h(x-u,y-v)\,W(du,dv)\,dxdy,
\end{equation}
and conversely, that if \eqref{mar} defines a semimartingale on $[0, T_1]\times[0,T_2]$ for some $T_1,T_2 > 0$, then $g$
satisfies these properties.

In our case, we define $g(s,t)=0$ for $s,\,t\leq 0$ and
\begin{equation}\label{}
  g(s,t):=\int_0^s\int_0^tx^{H_1-\frac{3}{2}}y^{H_2-\frac{3}{2}}e^{-\lambda_1x}e^{-\lambda_1 y}\,dxdy\,\,\, \text{for $s,\,t>0$}.
\end{equation}
Following as in Lemma \ref{lemmma}, we can show that the function $ g(t-u,s-v)-g(-u,s-v)-g(t-u,-v)+g(-u,-v)$ is square integrable over $\R^2$ for any $H_1,\,H_2>\frac{1}{2}$ and $\lambda_1,\,\lambda_2>0$.\\
Next, we observe that \eqref{hg} holds with $C=0$, $h(x,y)=0$ for $x,\,y<0$ and
\begin{equation}\label{}
  h(x,y):=x^{H_1-\frac{3}{2}}y^{H_2-\frac{3}{2}}e^{-\lambda_1x}e^{-\lambda_1 y}\in L^2(\R^2)
\end{equation}
for any $H_1,\,H_2>1$ and $\lambda_1,\,\lambda_2>0$. Then, it follows that the two-parameter tempered Hermite field of order one is a continuous semimartingale with canonical decomposition
\begin{align}
 &Z^{1,\,H_1,H_2}_{\lambda_1,\lambda_2}(s,t)\nonumber\\
 &\quad= \int_{\mathbb{R}^2}^\prime\int_0^t \int_0^s (a-x)_+^{H_1-\frac{3}{2}}e^{-\lambda_1(a-x)_+}(b-y)_+^{H_2-\frac{3}{2}}e^{-\lambda_2(b-y)_+}\,da\,db\,dW(x, y)\nonumber\\
 &\quad=\int_0^t \int_0^s\int_{\mathbb{R}^2}^\prime (a-x)_+^{H_1-\frac{3}{2}}e^{-\lambda_1(a-x)_+}(b-y)_+^{H_2-\frac{3}{2}}e^{-\lambda_2(b-y)_+}\,dW(x, y)\,da\,db,
\end{align}
which reduces to \eqref{38}. Since $C = 0$, a extension of \cite[Theorem 3.9]{Che2004:Gaussian} implies that $\{ Z^{1,\,H_1,H_2}_{\lambda_1,\lambda_2}\}$ is a
finite variation process.
\end{proof}
\begin{remark}
It is not hard to check that the two-parameter tempered Hermite field of order one is
not a semimartingale in the remaining case when $\frac{1}{2}< H_1,\,H_2 < 1$.
\end{remark}


\bibliographystyle{abbrv}
\bibliography{bibatef1.bib}
\end{document}